\documentclass[11pt]{birkjour}

\usepackage{amssymb}
\usepackage{eucal}
\usepackage{graphicx}
\usepackage{color}
\newtheorem{thm}{Theorem}[section]

\newtheorem{lem}[thm]{Lemma}
\newtheorem{prop}[thm]{Proposition}
\theoremstyle{definition}

\theoremstyle{remark}

\begin{document}
 \author[Omar Ajebbar and Elhoucien Elqorachi ]{Omar Ajebbar $^{2}$ and Elhoucien Elqorachi $^{1}$}
	
	\address{%
		$^1$
	 Ibn Zohr University, Faculty of sciences,
		Department of mathematics,
		Agadir,
		Morocco}
	\address{%
	$^2$
	Sultan Moulay Slimane University, Multidisciplinary Faculty,
	Department of mathematics and computer science,
	Beni Mellal,
	Morocco}
\email{omar-ajb@hotmail.com, elqorachi@hotmail.com }
\title[A generalization of the Cosine-Sine functional]{A generalization of the Cosine-Sine functional equation on a semigroup}
\begin{abstract}
Given a semigroup $S$ equipped with an
		involutive automorphism $\sigma$, we determine the complex-valued solutions $f,g,h$ of the functional equation
		\begin{equation*}f(x\sigma(y))=f(x)g(y)+g(x)f(y)+h(x)h(y),\,\,x,y\in S,\end{equation*}
		in terms of multiplicative functions and solutions of the special cases of sine and cosine-sine functional equations.
\end{abstract}
	\maketitle\emph{}
	\section{Introduction}Let $S$ be a semigroup. The cosine-sine functional equation is the functional equation
	\begin{equation}\label{eq0}f(xy)=f(x)g(y)+g(x)f(y)+h(x)h(y),\end{equation} for three unknown functions $f,g,h:S\to\mathbb{C}$.
	\par Chung, Kannappan and Ng \cite{CKN} solved
	the functional (\ref{eq0}) on groups. They expressed the solutions as linear combination of additive, multiplicative functions and their product whose coefficients enter into matrix relations.
	\par The authors \cite{AjElq} solved the following generalization of (\ref{eq0})
	\begin{equation}\label{eq1}f(x\sigma(y))=f(x)g(y)+g(x)f(y)+h(x)h(y),\end{equation}
	where $\sigma:S\to S$ is an involutive automorphism, on semigroups generated by their squares.
	\par In \cite{H.St1} Stetk{\ae}r solved the functional equation \begin{equation*}
	f(xy)=f(x)g(y)+g(x)f(y)-g(x)g(y)
	\end{equation*}
	on semigroups.
	\par Ebanks solved the functional equation (\ref{eq0}) on a t-compatible topological semigroup \cite [Theorem 5.1]{EB1}.
	\par Recently, the functional equation (\ref{eq0}) have been solved by Ebanks \cite{EB2}, and the authors \cite{AjElq2} on semigroups.
	\par The functional equation (\ref{eq1}) is a simple example of Levi-Civit\'{a}'s functional
	equation, and there is a general theory about the general form of
	the structure of solutions of Levi-Civit\'{a}'s functional equation
	on monoids, using matrix-coefficients of the right regular
	representation, see for example \cite[Theorem 5.2]{H.St}.
	We refer also to \cite{Acz} and \cite{H.St} for further contextual and historical discussions.\\The sine addition formula
	\begin{equation}\label{qo1}
	\phi (xy) = \phi (x)\chi(y) + \chi(x)\phi (y),\;x, y \in S,
	\end{equation}  and the functional equation
	\begin{equation}\label{qo2}
	f (xy) = f (x)\chi(y) + \chi(x)f (y)+\phi(x)\phi(y),\;x, y \in S,
	\end{equation}  where $\chi$ is a multiplicative function were recently solved by Ebanks \cite{E1}, \cite{E2} and  \cite{EB} on  semigroups.
	The  solutions of (\ref{qo1}) on abelian  groups, semigroups generated by it's squares was given by Vincze \cite{vin},   Ebanks and Stetkaer in \cite{stet1}.\\
	The theory developed for solving Levi–Civita functional equations is more comprehensive on groups than it is on semigroups, due to
	the existence of prime ideals in semigroups. The  solutions of (\ref{qo1}) on any group are
	exponential polynomials, whereas  this is not generally
	the case on in a semigroup generated by their squares. More precisely,  The solution of (\ref{qo1}) on a group
	contains a term $A\chi$, with $A$ additive and $\chi$ non-zero multiplicative function, on a semigroup generated by their squares  we
	may see instead the term
	$\phi(x)=B(x)\mu(x)$ if $\mu(x) \neq 0$, and $\phi(x) =
	0$
	if $\mu(x)=0$ or even more complicated (see for example \cite{Eba}, \cite{E1} and  \cite{EB}), where $B$ is an additive function defined on the subsemigroup, and where $\mu$ is non-zero multiplicative function. So there is an increase
	in complexity of solution forms of Levi-Civita equations as we move from the
	world of groups to the larger world of semigroups.
	\par Our main goal is to
	give the general solution of (\ref{eq1}) on all semigroups.  So, the new step forward is that we go from semigroups generated by their squares \cite{AjElq}   to general semigroup $S$.
	The subsemigroup $S^{2}$ will play a crucial role in processing of the functional equation (\ref{eq1}), and the essential role of the linear independence of the even part of the solutions $f$ and $h$  of (\ref{eq1})  is clear from our formulations.\par The solutions are expressed in terms of multiplicative functions and solutions of the special cases of sine (\ref{qo1}) and cosine-sine (\ref{qo2}) functional equations without disjoining their forms on partitions of $S$ using prime ideals corresponding to multiplicative functions.
	\par The organization of the paper is as follows. In Sect. 2, we give notations and terminology. In Sect. 3, we give preliminary results an key properties that we will need in the paper. We present  the current state of knowledge about equation (\ref{eq0}) in Proposition 3.2. We describe the properties of  the solutions of (\ref{eq1}) in five technical lemmas of the special cases:  $f^{e}\neq0$ on $S^{2}$ (Lemma 3.5); $f^{e}\neq0$ and  $f^{e}=0$ on $S^{2}$ (Lemma 3.6); and ($f^{e},h^{e}$) linearly independent (Lemma 3.7).  The
	main results are given in the fourth section. We  discuss two principal cases according to whether $f^{e}$ and $h^{e}$ are linearly independent or not. We also take into account the fact that $f^{e}$ is zero on  $S^{2}$ or not.
	\section{Notations and terminology}
	Throughout this paper $S$ denotes a semigroup. That is a set with an
	associative composition. The map $\sigma:S\to S$ denotes an involutive automorphism. That
	$\sigma$ is involutive means that $\sigma(\sigma(x))=x$ for all
	$x\in S$.\\
	We define $S^{2}:=\{xy\mid x\in S\,\text{and}\,y\in S\}.$\\
	Let $f:S\to \mathbb{C}$ be a function. We define $f^{*}:=f\circ\sigma$. We call $f^{e}:=\frac{f+f^{*}}{2}$ the even part of $f$ and $f^{o}:=\frac{f-f^{*}}{2}$ its odd part. The function $f$ is said to be even if $f^{*}=f$, and $f$ is said to be odd if $f^{*}=-f$.\\
	$f$ is said to be central if $f(xy)=f(yx)$ for all $x,y\in S$, $f$ is said to be additive if $f(xy)=f(x)+f(y)$ for all $x,y\in S$ and $f$ is said multiplicative if $f(xy)=f(x)f(y)$ for all $x,y\in S$.\\
	Let $T$ be a non empty subset of $S$. $f$ is zero on $T$ if $f(x)=0$ for all $x\in T$, we denote $f=0$ on $T$. $f$ is non-zero on $T$ if there exists an element $x\in T$ such that $f(x)\neq0$. We denote $f\neq0$ on $T$.\\
	Let $\chi,\varphi,\psi:S\to\mathbb{C}$ be functions such that $\chi$ is multiplicative. we say that\\
	$\varphi$ is $\chi$-additive if the pair $(\varphi,\chi)$ satisfies the special case of the sine addition law $\varphi(xy)=\varphi(x)\chi(y)+\chi(x)\varphi(y),\,\,x,y\in S.$ Notice that $\varphi$ is $1$-additive ($\chi=1$) means that $\varphi$ is additive, while $\varphi$ is $0$-additive ($\chi=0$) means that $\varphi$ is zero on $S^{2}$.\\
	$\psi$ is of type $(\chi,\sigma,\varphi)$-cosine-sine if the triple $(\psi,\chi,\varphi)$ satisfies the special case of the cosine-sine functional equation $\psi(x\sigma(y))=\psi(x)\chi(y)+\chi(x)\psi(y)+\varphi(x)\varphi(y)\,\,x,y\in S,$ where $\varphi$ is $\chi$-additive.\\
	$\psi$ is of type $(\chi,\varphi)$-cosine-sine if it is of type $(\chi,I_{S},\varphi)$-cosine-sine, where $I_{S}$ is the identity function on $S$.
\section{Preliminary results}
\subsection{About multiplicative functions, sine addition law and cosine-sine functional equation}
\begin{prop}\label{prop001}
Let $S$ be a semigroup and $\varphi:S\to\mathbb{C}$ be a $\chi$-additive function.
\par If $\varphi=\Sigma_{j=1}^{N}c_{j}\chi_{j}$, where
$c_{j}\in\mathbb{C}$ and $\chi_{j}:S\to\mathbb{C}$ is
multiplicative for each $j=1,2,\ldots{},N$, then $\varphi=0$.
\end{prop}
\begin{proof} Without loss of generality we assume that $\chi\neq 0$, and that $\chi_{1},...,\chi_{N}$ are non-zero and different. Let $x,y\in S$ be arbitrary. We have $$\varphi(xy)=\varphi(x)\chi(y)+\chi(x)\varphi(y)=\Sigma_{j=1}^{N}c_{j}\chi_{j}(xy)$$ for all $x,y\in S$. Then $$\varphi(y)\chi+\Sigma_{j=1}^{N}c_{j}(\chi(y)-\chi_{j}(y))\chi_{j}=0$$ for each arbitrary $y\in S$.
If $\chi_{j}=\chi$ for some $j$, then the corresponding term
$c_{j}(\chi(y)-\chi_{j}(y))\chi_{j}$ of the identity below vanishes, so
the multiplicative function $\chi$ does not occur in the sum. According
to \cite[Theorem 3.18(b)]{H.St} we obtain from the identity below that
$\varphi(y)=0$. So, $y$ being arbitrary, we deduce that $\varphi=0$. This
completes the proof of Proposition \ref{prop001}.
\end{proof}
\begin{prop}\label{prop1}\cite[Proposition 3.4]{AjElq2} Let $f,g,h:S\to\mathbb{C}$ be solutions of the functional equation
$$f(xy)=f(x)g(y)+g(x)f(y)+h(x)h(y),\,\,x,y\in S,$$ such that $f$ and $h$ are linearly independent. Then the triple $(f,g,h)$ satisfies one of the following possibilities:\\
(1) There exist a constant $\beta\in\mathbb{C}$ and two multiplicative functions $m,\chi:S\to\mathbb{C}$ such that
\begin{equation*}g=\chi,\,h-\beta f=\varphi\,\,\text{and}\,\,\chi+\beta h=m,\end{equation*}
where $\varphi:S\to\mathbb{C}$ is $\chi$-additive\\
(2) There exist constants $\alpha\in\mathbb{C}\setminus\{0\}, \beta\in\mathbb{C}$ and a multiplicative function $\chi:S\to\mathbb{C}$ such that
\begin{equation*}-\alpha\,f+g=\chi\end{equation*}
and the pair $(h,\chi+2\,\alpha f+\dfrac{1}{2}\,\beta h)$ satisfies the sine addition law, i.e.,
\begin{equation*}h(xy)=h(x)(\chi+2\,\alpha f+\dfrac{1}{2}\,\beta h)(y)+(\chi+2\,\alpha f+\dfrac{1}{2}\,\beta h)(x)h(y),\,\,x,y\in S.\end{equation*}
(3) $$f=F,\,\,\,\,g=-\dfrac{1}{2}\delta^{2}\,F+G+\delta\,H,\,\,\,\,h=-\delta
\,F+H,$$
where $\delta\in\mathbb{C}$ is a constant and $F,G,H:S\to\mathbb{C}$ are functions such that the triple $(F,G,H)$ satisfies (1)-(2).
\end{prop}
\subsection{Key properties of solutions}
\begin{lem}\label{lem0}Let $f,g,h:S\to\mathbb{C}$ be a solution of (\ref{eq1}). Then\\
(1) $f^{e}(xy)=f^{e}(yx)$ for all $x,y\in S$, i.e., $f^{e}$ is central.\\
(2) $f^{o}(xy)=-f^{o}(yx)$ for all $x,y\in S$.\\
(3) $f^{o}(xyz)=0$ for all $x,y,z\in S$.\\
(4)\begin{equation}\label{eq2}f^{o}(xy)+f^{o}(x\sigma(y))=2f^{o}(x)g^{e}(y)+2g^{o}(x)f^{e}(y)+2h^{o}(x)h^{e}(y)\end{equation}
and
\begin{equation}\label{eq3}f^{e}(xy)+f^{e}(x\sigma(y))=2f^{e}(x)g^{e}(y)+2g^{e}(x)f^{e}(y)+2h^{e}(x)h^{e}(y)\end{equation}
for all $x,y\in S$.
\end{lem}
\begin{proof}
(1) The right hand side of the functional equation
(\ref{eq1}) is invariant under the interchange of $x$ and $y$. So
$f(x\sigma(y))=f(y\sigma(x))$ for all $x,y\in S$. Then $f^{*}(xy)=f(\sigma(x)\sigma(y))=f(y\sigma(\sigma(x)))=f(yx)$ for all $x,y\in S$.\\
(2) From (1) we get that $f^{o}(xy)=f(xy)-f(yx)$ for all $x,y\in S$. Then $f^{o}(xy)=-f^{o}(yx)$ for all $x,y\in S$.\\
(3) Let $x,y,z\in S$ be arbitrary. Using (2) and the associativity of the composition in $S$, we get that $f^{o}(xyz)=f^{o}(x(yz))=-f^{o}(y(zx))=f^{o}(z(xy))=-f^{o}(xyz)$. So, $f^{o}(xyz)=0$ for all $x,y,z\in S$.\\
(4) Let $x,y\in S$ be arbitrary. Applying (\ref{eq1}) to the pairs $(\sigma(x),\sigma(y)),\,(x,\sigma(y))$ and $(\sigma(x),y)$ we obtain
\begin{equation}\label{eq4}f^{*}(x\sigma(y))=f^{*}(x)g^{*}(y)+g^{*}(x)f^{*}(y)+h^{*}(x)h^{*}(y),\end{equation}
\begin{equation}\label{eq5}f(xy)=f(x)g^{*}(y)+g(x)f^{*}(y)+h(x)h^{*}(y)\end{equation}
and
\begin{equation}\label{eq6}f^{*}(xy)=f^{*}(x)g(y)+g^{*}(x)f(y)+h^{*}(x)h(y).\end{equation}
By subtracting (\ref{eq6}) from (\ref{eq5}), and (\ref{eq4}) from (\ref{eq1}) and adding the two obtained identities we get (\ref{eq2}) by a small computation. Similarly, we derive (\ref{eq3}) by adding (\ref{eq1}), (\ref{eq4}), (\ref{eq5}) and (\ref{eq6}). This completes the proof of Lemma \ref{lem0}.
\end{proof}
\begin{lem}\label{lem1} Let $f,g,h:S\to\mathbb{C}$ be a solution of (\ref{eq1}) such that $f\neq0$ and $f^{*}=-f$. Then\\
(1) $g^{*}=-g$ and $h^{*}=-h$.\\
(2) $f(xy)=0$ for all $x,y\in S$ and there exists a constant $\lambda\in\mathbb{C}$ such that $g=-\frac{\lambda^{2}}{2}f$ and $h=\lambda f$.
\end{lem}
\begin{proof} From (\ref{eq3}) we deduce that $h^{e}(x)h^{e}(y)=0$ for all $x,y\in S$. Then $h^{e}=0$ which implies that $h^{*}=-h$.\\
Let $x,y\in S$ be arbitrary. Since $f^{*}=-f$ and $h^{*}=-h$ we get from the functional equation (\ref{eq2}) that
\begin{equation}\label{eq8}f(xy)+f(x\sigma(y))=2f(x)g^{e}(y).\end{equation}
Using Lemma \ref{lem0}(2) we have $f(xy)=-f(yx)$ and $f(x\sigma(y))=-f(\sigma(y)x)=f(y\sigma(x))$. So, by writing (\ref{eq8}) for the pair $(y,x)$ and adding the identity obtained to (\ref{eq8}) we derive obtain
\begin{equation}\label{eq7}f^{o}(x\sigma(y))=f^{o}(x)g^{e}(y)+g^{e}(x)f^{o}(y)\end{equation} for all $x,y\in S.$\\
Now, applying (\ref{eq7}) to the pair $(xy,z)$ and using Lemma \ref{lem0}(3) we get that $$f^{o}(xy)g^{e}(z)+g^{e}(xy)f^{o}(z)=0$$ for all $x,y,z\in S.$ Then $\{f^{o},g^{e}\}$ is linearly dependent. Indeed if $\{f^{o},g^{e}\}$ is linearly independent we deduce from the identity above that $f^{o}(xy)=0$ for all $x,y\in S.$  By using (\ref{eq7}) we get that $f^{o}(x)=0$ for all $x\in S$ and then $f=f^{o}=0$, which contradicts the assumption on $f.$ So, there exists a constant $c\in\mathbb{C} $ such that $g^{e}=cf^{o}$ because $f^{o}\neq0$. As $g^{e}$ is even and $f^{o}$ is odd we get that $g^{e}=0.$ Then $f(xy)=f^{o}(xy)=0$ for all $x,y\in S$ and $g^{*}=-g.$\\
Hence, (\ref{eq1}) reduces to
\begin{equation}\label{eq9}f(x)g(y)+g(x)f(y)=-h(x)h(y)\end{equation}
for all $x,y\in S$.
\par If $h\neq0$. Then, we infer from (\ref{eq9}) that there exist constants $\alpha,\beta\in\mathbb{C}$ such that $h=\alpha f+\beta g$. Substituting this in (\ref{eq9}) we obtain after a small computation that $$f(x)\left((1+\alpha\beta)g(y)+\alpha^{2}f(y)\right)+g(x)\left((1+\alpha\beta)f(y)+\beta^{2}g(y)\right)=0$$
for all $x,y\in S.$ So, $f$ and $g$ are linearly dependent. Indeed, if not we deduce from the identity above that $1+\alpha\beta=\alpha^{2}=\beta^{2}=0$, which is a contradiction. As $f\neq0$ there exists a constant $\gamma\in\mathbb{C}$ such that $g=\gamma f$. As $h=\alpha f+\beta g$ we derive that there exists a constant $\lambda\in \mathbb{C}$ such that $h=\lambda f$. So that (\ref{eq9}) implies that $2\gamma f(x)f(y)=-\lambda^{2} f(x)f(y)$ for all $x,y\in S$. As $f\neq0$ we deduce that $\gamma=-\frac{\lambda^{2}}{2}$. Hence $g=-\frac{\lambda^{2}}{2} f$.
\par If $h=0$. Then we get from (\ref{eq9}) that $f(x)g(y)=-g(x)f(y)$ for all $x,y\in S$. Hence $g=0$ because $f\neq0$. So, we get a special case of part (2) corresponding to $\lambda=0$. This completes the proof of Lemma \ref{lem1}.
\end{proof}
\subsection{Links with fundamental functional equations}
\begin{lem}\label{lem2} Let $f,g,h:S\to\mathbb{C}$ be a solution of (\ref{eq1}) such that $f^{e}\neq0$ on $S^{2}$ and there exists a constant $\lambda\in\mathbb{C}$ such that $h^{e}=\lambda f^{e}$. Then there exists a constant $\mu\in\mathbb{C}$ such that
\begin{equation}\label{eq000}g^{o}+\lambda h^{o}=\mu f^{o}.\end{equation} Moreover one of the following possibilities holds.\\
(1) $f^{o}=0$, $g^{o}=-\lambda h^{o}$ and the triple $(f^{e},g^{e}+\frac{\lambda^{2}}{2}f^{e},h^{o})$ satisfies the cosine-sine addition law, i.e.,
\begin{equation}\label{eq-11}f^{e}(xy)=f^{e}(x)\left(g^{e}+\frac{\lambda^{2}}{2}f^{e}\right)(y)+f^{e}(y)\left(g^{e}+\frac{\lambda^{2}}{2}f^{e}\right)(x)-h^{o}(x)h^{o}(y)\end{equation}
for all $x,y\in S.$\\
(2) $g^{e}=-\mu f^{e}$ for some constant $\mu\in\mathbb{C}$,
\begin{equation}\label{eq013}f^{o}(xy)=0,\end{equation}
\begin{equation}\label{eq14}f^{e}(xy)+f^{e}(\sigma(y)x)=2(\lambda^{2}-2\mu)f^{e}(x)f^{e}(y)\end{equation} and
\begin{equation}\label{eq140}\begin{split}&f^{e}(x\sigma(y))=(\lambda^{2}-2\mu)f^{e}(x)f^{e}(y)+2\mu f^{o}(x)f^{o}(y)\\
 &\quad\quad\quad\quad\quad\,\,-\lambda f^{o}(x)h^{o}(y)-\lambda h^{o}(x)f^{o}(y)+h^{o}(x)h^{o}(y)\end{split}\end{equation}
 for all $x,y\in S.$
\end{lem}
\begin{proof}
Since $h^{e}=\lambda f^{e}$ the functional equation (\ref{eq2}) reduces to
\begin{equation}\label{eq02}f^{o}(xy)+f^{o}(x\sigma(y))=2f^{o}(x)g^{e}(y)+2\left(g^{o}+\lambda h^{o}\right)(x)f^{e}(y),\,\,x,y\in S.\end{equation}
As $f^{o}(xyz)=$ for all $x,y,z\in S$ we deduce from (\ref{eq02}) that $$f^{o}(x)g^{e}(y)+\left(g^{o}+\lambda h^{o}\right)(x)f^{e}(y)=0$$ for all $x\in S,\, y\in S^{2}$. Since $f^{e}\neq0$ on $S^{2}$ we get from the identity above that there exists a constant $\mu\in\mathbb{C}$ such that $g^{o}+\lambda h^{o}=\mu f^{o}$, which is (\ref{eq000}). Hence, (\ref{eq02}) becomes
\begin{equation}\label{eq03}f^{o}(xy)+f^{o}(x\sigma(y))=2f^{o}(x)(g^{e}(y)+\mu f^{e}(y)),\,\,x,y\in S.\end{equation}
Then, in view of Lemma \ref{lem0}(3), we get that $f^{o}(xy)(g^{e}(z)+\mu f^{e}(z))=0$ for all $x,y,z\in S$. So that $f^{o}=0$ on $S^{2}$ or $g^{e}=-\mu f^{e}$. Hence, taking (\ref{eq03}) into account, we deduce that $f^{o}(x)(g^{e}(y)+\mu f^{e}(y))=0$ for all $x,y\in S$. So, we discuss according to whether $f^{o}=0$ or $g^{e}=-\mu f^{e}$.\\
\underline{Case A}: $f^{o}=0$. Then $f=f^{e}$ and in view of (\ref{eq000}) we get that $g^{o}=-\lambda h^{o}$. So, by applying (\ref{eq1}) to the pair $(x,\sigma(y))$ and subtracting (\ref{eq1}) from the identity obtained, using that $h^{e}=\lambda f^{e}$, a small computation shows that
\begin{equation}\label{eq10}f(xy)-f(x\sigma(y))=-2h^{o}(x)h^{o}(y)\end{equation} for all $x,y\in S.$
When we add this to (\ref{eq3}) the result leads to (\ref{eq-11}). The result occurs in part (1).\\
\underline{Case B}: $g^{e}=-\mu f^{e}$. So, from (\ref{eq03}) we get
\begin{equation*}f^{o}(xy)+f^{o}(x\sigma(y))=0,\,\,x,y\in S.\end{equation*} By using (\ref{eq3}) and that $g^{e}=-\mu f^{e}$ and $h^{e}=\lambda f^{e}$ we obtain (\ref{eq14}).
By adding the identity above and  (\ref{eq14}) we get that
\begin{equation}\label{eq014}f(xy)+f(x\sigma(y))=2(\lambda^{2}-2\mu)f^{e}(x)f^{e}(y)\end{equation} for all $x,y\in S.$ The right hands of (\ref{eq1}) and (\ref{eq014}) are invariant under the interchange of $x$ and $y$. So $f(xy)=f(yx)$ for all $x,y\in S$, i.e., $f$ is central. Then, taking Lemma \ref{lem0}(1)(2) into account, we get (\ref{eq013}).\\
On the other hand, using (\ref{eq1}) and that any complex-valued function on $S$ is the sum of its even and odd parts we deduce
\begin{equation*}\begin{split}&f(x\sigma(y))=f^{e}(x)g^{e}(y)+g^{e}(x)f^{e}(y)+h^{e}(x)h^{e}(y)\\
 &\quad\quad\quad\quad\quad+f^{o}(x)g^{o}(y)+g^{o}(x)f^{o}(y)+h^{o}(x)h^{o}(y)\\
 &\quad\quad\quad\quad\quad+f^{e}(x)g^{o}(y)+g^{e}(x)f^{o}(y)+h^{e}(x)h^{o}(y)\\
 &\quad\quad\quad\quad\quad+f^{o}(x)g^{e}(y)+g^{o}(x)f^{e}(y)+h^{o}(x)h^{e}(y)\end{split}\end{equation*}
 for all $x,y\in S.$ Now, using that $g^{e}=-\mu f^{e}$, $h^{e}=\lambda f^{e}$ and $g^{o}+\lambda h^{o}=\mu f^{o}$ the computation above reduces to (\ref{eq140}). This completes the proof of Lemma \ref{lem2}.
\end{proof}
\begin{lem}\label{lem3} Let $f,g,h:S\to\mathbb{C}$ be a solution of (\ref{eq1}) such that $f^{e}\neq0$, $f^{e}=0$ on $S^{2}$ and there exists a constant $\lambda\in\mathbb{C}$ such that $h^{e}=\lambda f^{e}$. Then \\
(1) $g^{e}=-\frac{\lambda^{2}}{2} f^{e}$.\\
(2)
\begin{equation}\label{eq15}f^{o}(x\sigma(y))=k(x)f^{e}(y)+f^{e}(x)k(y),\,\,x,y\in S,\end{equation}
where $k:=-\frac{\lambda^{2}}{2} f^{o}+g^{o}+\lambda h^{o}.$\\
(3)
\begin{equation}\label{eq16}f^{o}(x)g^{o}(y)+g^{o}(x)f^{o}(y)+h^{o}(x)h^{o}(y)=0,\,\,x,y\in S.\end{equation}
\end{lem}
\begin{proof}(1) Let $x,y\in S$ be arbitrary. Since $f^{e}=0$ on $S^{2}$  and $h^{e}=\lambda f^{e}$ we get from (\ref{eq3}) that $f^{e}(x)\left(g^{e}(y)+\frac{\lambda^{2}}{2} f^{e}(y)\right)=-f^{e}(y)\left(g^{e}(x)+\frac{\lambda^{2}}{2} f^{e}(x)\right)$. So, $x,y\in S$ being arbitrary and $f^{e}\neq0$, we deduce that $g^{e}=-\frac{\lambda^{2}}{2} f^{e}$. This is part (1).\\
(2) Using that $h^{e}=\lambda f^{e}$ and part(1) the functional equation (\ref{eq2}) becomes $$f^{o}(xy)+f^{o}(x\sigma(y))=2f^{o}(x)\left(-\frac{\lambda^{2}}{2} f^{e}(y)\right)+2g^{o}(x)f^{e}(y)+2\lambda h^{o}(x)f^{e}(y)$$ for all $x,y\in S$, which implies that
\begin{equation}\label{eq160}f^{o}(xy)+f^{o}(x\sigma(y))=2k(x)f^{e}(y)\end{equation} for all $x,y\in S$ where $k:=-\frac{\lambda^{2}}{2} f^{o}+g^{o}+\lambda h^{o}.$
Now, applying (\ref{eq160}) to the pair $(y,x)$, using Lemma \ref{lem0}(2) and that $f^{o}$ is odd we obtain $$-f^{o}(xy)+f^{o}(x\sigma(y))=2k(y)f^{e}(x).$$ So, adding this to (\ref{eq160}) we get part (2).\\
(3) Since $f^{e}=0$ on $S^{2}$ we have $f(x\sigma(y))=f^{o}(x\sigma(y))$ for all $x,y\in S$. So, from (\ref{eq1}) and (\ref{eq15}) we deduce that $$f(x)g(y)+g(x)f(y)+h(x)h(y)=k(x)f^{e}(y)+f^{e}(x)k(y)$$ for all $x,y\in S.$ As $f=f^{e}+f^{o}$, $g=g^{e}+g^{o}=-\frac{\lambda^{2}}{2} f^{e}+g^{o}$, $h=h^{e}+h^{o}=\lambda f^{e}+h^{o}$ and $k=-\frac{\lambda^{2}}{2} f^{o}+g^{o}+\lambda h^{o}$ we derive from the identity above that
\begin{equation*}\begin{split}&-\frac{\lambda^{2}}{2} f^{e}(x)f^{e}(y)+f^{e}(x)g^{o}(y)-\frac{\lambda^{2}}{2} f^{o}(x)f^{e}(y)+f^{o}(x)f^{o}(y)\\
&\quad\quad-\frac{\lambda^{2}}{2} f^{e}(x)f^{e}(y)-\frac{\lambda^{2}}{2} f^{e}(x)f^{o}(y)+g^{o}(x)f^{e}(y)+g^{o}(x)f^{o}(y)\\
&\quad\quad+\lambda^{2}f^{e}(x)f^{e}(y)+\lambda f^{e}(x)h^{o}(y)+\lambda h^{o}(x)f^{e}(y)+h^{o}(x)h^{o}(y)\\
&=-\frac{\lambda^{2}}{2} f^{o}(x)f^{e}(y)+g^{o}(x)f^{e}(y)+\lambda h^{o}(x)f^{e}(y)-\frac{\lambda^{2}}{2} f^{o}(y)f^{e}(x)\\
&\quad\quad+g^{o}(y)f^{e}(x)+\lambda h^{o}(y)f^{e}(x)\end{split}\end{equation*} for all $x,y\in S$, which implies that
$f^{o}(x)g^{o}(y)+g^{o}(x)f^{o}(y)+h^{o}(x)h^{o}(y)=0$
for all $x,y\in S.$ The result occurs in part (3). This completes the proof of Lemma \ref{lem3}.
\end{proof}
In Lemma \ref{lem4} we establish, when $\{f^{e},h^{e}\}$ is linearly independent, a link between the cosine-sine functional equation and the functional equation (\ref{eq1}) and generalized version of the multiplicative Cauchy equation .
\begin{lem}\label{lem4} Let $f,g,h:S\to\mathbb{C}$ be a solution of (\ref{eq1}) with $f^{e}$ and $h^{e}$ independent. Then one of the following possibilities holds.\\
(1) $f$, $g$ and $h$ are even functions such that $\{f,g,h\}$ is linearly independent and $(f,g,h)$ satisfies the cosine-sine addition law
\begin{equation}\label{eq17}f(xy)=f(x)g(y)+g(x)f(y)+h(x)h(y),\,\,x,y\in S.\end{equation}
(2) $\{f^{e},g^{e},h^{e}\}$ is linearly dependent and $(f^{e},g^{e},h^{e})$ satisfies the cosine-sine addition law
\begin{equation}\label{eq--18}f^{e}(xy)=f^{e}(x)g^{e}(y)+g^{e}(x)f^{e}(y)+h^{e}(x)h^{e}(y),\,\,x,y\in S,\end{equation}
$f^{o}(xy)=0$ for all $x,y\in S$, $g^{e}=\frac{\beta^{2}}{2}f^{e}+\beta h^{e}$, $g^{o}=-\frac{\beta^{2}}{2}f^{o}$ and $h^{o}=-\beta f^{o}$ for some constant $\beta\in\mathbb{C}.$\\
(3) There exists a constant $\beta\in\mathbb{C}$ such that $\beta f^{o}+h^{o}\neq0$, $g=\frac{\beta^{2}}{2}f+\beta h$ and the pair $(f,l)$ satisfies the functional equation
\begin{equation}\label{eq+17}f(x\sigma(y))=l(x)l(y),\,\,x,y\in S,\end{equation}
where $l:=\beta f+h.$
\end{lem}
\begin{proof} $\{f^{e},h^{e}\}$ is linearly independent. So we split the discussion according to $\{f^{e},g^{e},h^{e}\}$ is linearly independent or not.\\
\underline{Case A}: $\{f^{e},g^{e},h^{e}\}$ is linearly independent. Let $x,y,z\in S$ be arbitrary. By  applying (\ref{eq2}) to the pair $(xy,z)$, and using Lemma \ref{lem0}(3), we get that $f^{o}(xy)g^{e}(z)+g^{o}(xy)f^{e}(z)+h^{o}(xy)h^{e}(z)=0.$
So, $x,y,z\in S$ being arbitrary and $\{f^{e},g^{e},h^{e}\}$ linearly independent, we derive that $f^{o}(xy)=0$ for all $x,y\in S.$ Hence, we get from (\ref{eq2}) that $f^{o}(x)g^{e}(y)+g^{o}(x)f^{e}(y)+h^{o}(x)h^{e}(y)=0$ for all $x,y\in S$. So, by using that $\{f^{e},g^{e},h^{e}\}$ is linearly independent, we deduce that $f^{o}(x)=g^{o}(x)=h^{o}(x)=0$ for all $x\in S.$ Then $f$, $g$ and $h$ are even functions. So that, by applying (\ref{eq1}) to the pair $(x,\sigma(y))$ we obtain (\ref{eq17}). Hence, we get part (1).\\
\underline{Case B}: $\{f^{e},g^{e},h^{e}\}$ is linearly dependent. Then there exist $\alpha,\beta\in\mathbb{C}$ such that
\begin{equation}\label{eq-18}g^{e}=\alpha f^{e}+\beta h^{e}.\end{equation} So, from (\ref{eq2}) and (\ref{eq3}) we derive respectively that
\begin{equation}\label{eq18}f^{o}(xy)+f^{o}(x\sigma(y))=2F_{1}(x)f^{e}(y)+2F_{2}(x)h^{e}(y),\,\,x,y\in S,\end{equation}
where
\begin{equation}\label{eq19}F_{1}:=\alpha f^{o}+g^{o}\,\,\text{and}\,\,F_{2}:=\beta f^{o}+h^{o}\end{equation}
and
\begin{equation}\label{eq20}\begin{split}&f^{e}(xy)+f^{e}(x\sigma(y))\\
&\quad\quad\quad\quad\quad\quad=2\left((2\alpha-\beta^{2})f^{e}(x)f^{e}(y)+(2\beta f^{e}+h^{e})(x)(2\beta f^{e}+h^{e})(y)\right)\end{split}\end{equation}
for all $x,y\in S.$\\
Let $x,y,z\in S$ be arbitrary. By applying (\ref{eq18}) to each of the pairs $(xy,z)$, and $(x,yz)$ and taking Lemma \ref{lem0}(3) into account, we get that
\begin{equation}\label{eq21}F_{1}(xy)f^{e}(z)+F_{2}(xy)h^{e}(z)=0\end{equation}
and
\begin{equation}\label{eq22}F_{1}(x)f^{e}(yz)+F_{2}(x)h^{e}(yz)=0.\end{equation}
So, $x,y,z\in S$ being arbitrary and $\{f^{e},h^{e}\}$ linearly independent, we infer from (\ref{eq21}) that $F_{1}(xy)=F_{2}(xy)=0$ for all $x,y\in S$. Then, in view of (\ref{eq19}) we obtain
\begin{equation}\label{eq23}g^{o}(xy)=-\alpha f^{o}(xy)\,\,\text{and}\,\,h^{o}(xy)=-\beta f^{o}(xy)\end{equation}
for all $x,y\in S$.\\
On the other hand, since $\{f^{e},h^{e}\}$ is linearly independent so is the pair $\{f^{e},2\beta f^{e}+h^{e}\}$. Then, if $f^{e}(xy)=0$ for all $x,y\in S$, we deduce by applying (\ref{eq20}), that $2\beta f^{e}+h^{e}=0$, which contradicts that $\{f^{e},h^{e}\}$ is linearly independent. So, $f^{e}\neq0$ on $S^{2}$. Hence, we deduce from (\ref{eq22}) that
\begin{equation}\label{eq24}F_{1}=\lambda F_{2}\end{equation}
for some constant $\lambda\in\mathbb{C}.$ We split the discussion into the subcases $F_{2}=0$ and $F_{2}\neq0$.\\
\underline{Subcase B1}: $F_{2}=0$. Then, in view of (\ref{eq24}) and (\ref{eq19}) we get that
\begin{equation}\label{eq25}g^{o}=-\alpha f^{o}\,\,\text{and}\,\,h^{o}=-\beta f^{o}.\end{equation}
Moveover, the identity (\ref{eq18}) implies that $$f^{o}(xy)+f^{o}(x\sigma(y))=0\,\,\text{and}\,\,f^{o}(yx)+f^{o}(y\sigma(x))=0$$ for all $x,y\in S.$
So, using Lemma \ref{lem0}(2) and that $f^{o}\circ\sigma=-f^{o},$ and adding the two identities above we obtain
\begin{equation}\label{eq26}f^{o}(xy)=0\end{equation}
for all $x,y\in S,$ i.e., $f^{o}=0$ on $S^{2}$. Hence, $f(x\sigma(y))=f^{e}(x\sigma(y))$ for all $x,y\in S$ and we get from the functional equation (\ref{eq1}) that
\begin{equation*}\begin{split}&f^{e}(x\sigma(y))\\
&=f(x)g(y)+g(x)f(y)+h(x)h(y)\\
&=\left(f^{e}(x)+f^{o}(x)\right)\left(g^{e}(y)+g^{o}(y)\right)+\left(g^{e}(x)+g^{o}(x)\right)\left(f^{e}(y)+f^{o}(y)\right)\\
&\quad+\left(h^{e}(x)+h^{o}(x)\right)\left(h^{e}(y)+h^{o}(y)\right)\end{split}\end{equation*}
for all $x,y\in S,$ which implies, by using (\ref{eq25}), that
\begin{equation*}\begin{split}&f^{e}(x\sigma(y))=f^{e}(x)g^{e}(y)+f^{e}(y)g^{e}(x)+f^{o}(y)\left(g^{e}(x)-\alpha f^{e}(x)-\beta h^{e}(x)\right)\\
&\quad\quad\quad\quad\quad\,+f^{o}(x)\left(g^{e}(y)-\alpha f^{e}(y)-\beta h^{e}(y)\right)+h^{e}(x)h^{e}(y)\\
&\quad\quad\quad\quad\quad\,+(\beta^{2}-2\alpha)f^{o}(x)f^{o}(y)\end{split}\end{equation*}
for all $x,y\in S.$\\
Now, (\ref{eq-18}) applied to the pair $(x,\sigma(y))$ leads to
\begin{equation}\label{eq27}f^{e}(xy)=(2\alpha-\beta^{2})f^{e}(x)f^{e}(y)+H(x)H(y)-(\beta^{2}-2\alpha)f^{o}(x)f^{o}(y)\end{equation}
for all $x,y\in S,$ where $H:=\beta f^{e}+h^{e}.$\\
Let $x,y,z\in S$ be arbitrary. By applying the identity above to the pairs $(xy,z)$ and $(zx,y)$, using the associativity of the composition in $S$, that $f^{e}$ is central and $f^{o}=0$ on $S^{2}$, comparison of the results leads to $$(2\alpha-\beta^{2})f^{e}(xy)f^{e}(z)+H(xy)H(z)=(2\alpha-\beta^{2})f^{e}(xz)f^{e}(y)+H(xz)H(y)$$
for all $x,y,z\in S.$\\
 Since $\{f^{e},h^{e}\}$ is linearly independent so is $\{f^{e},H\}$. Then there exist $z_{1},z_{2}\in S$ such that $f^{e}(z_{1})H(z_{2})-f^{e}(z_{2})H(z_{1})\neq 0$.\\
By fixing $z=z_{1}$ and $z=z_{2}$ in the identity above we get that
$$(2\alpha-\beta^{2})f^{e}(xy)=k_{1}(x)f^{e}(y)+k_{2}(x)H(y)$$
for all $x,y\in S.$  Then, using that $f^{e}$ and $H$ are even functions we deduce that $(2\alpha-\beta^{2})f^{e}(xy)=(2\alpha-\beta^{2})f^{e}(x\sigma(y))$ for all $x,y\in S.$ Combining this with (\ref{eq27}) and seeing that $f^{e}\circ\sigma=f^{e}$, $H^{*}=H$ and $f^{o}\circ\sigma=-f^{o}$ we derive that $(2\alpha-\beta^{2})^{2}f^{o}(x)f^{o}(y)=0$ for all $x,y\in S.$ So that
$(2\alpha-\beta^{2})f^{o}(x)=0$ for all $x,y\in S.$ Then $2\alpha=\beta^{2}$ or $f^{o}=0.$
\par If $2\alpha=\beta^{2}.$ Then, according to (\ref{eq27}) and (\ref{eq-18}) we deduce respectively that $(f^{e},g^{e},h^{e})$ satisfies the cosine-sine addition law
$$f^{e}(xy)=f^{e}(x)g^{e}(y)+g^{e}(x)f^{e}(y)+h^{e}(x)h^{e}(y),\,\,x,y\in S$$
and $g^{e}=\frac{\beta^{2}}{2}f^{e}+\beta h^{e},$ while the odd part $f^{o}$ satisfies (\ref{eq26}), and according to (\ref{eq25}) we have $g^{o}=-\frac{\beta^{2}}{2}f^{o}$ and $h^{o}=-\beta f^{o}.$ The result occurs in part (2).
\par If $f^{o}=0.$ Then, from (\ref{eq25}) we deduce that $g^{o}=h^{o}=0$.\\ So, as in Case A we get the result in part (1).\\
\underline{Subcase B2}: $F_{2}\neq0$. Let $x,y,z\in S$ be arbitrary. By applying (\ref{eq1}) to the pairs $(x,\sigma(yz))$ and $(x,\sigma(zy))$, and subtracting the identities obtained we get that
\begin{equation*}\begin{split}&f(xyz)-f(xzy)\\
&=f(x)\left(g^{*}(yz)-g^{*}(zy)\right)+g(x)\left(f^{*}(yz)-f^{*}(zy)\right)+h(x)\left(h^{*}(yz)-h^{*}(zy)\right).\end{split}\end{equation*}
Since $f^{e}$ is central and $F_{2}\neq0$ we derive from  (\ref{eq22}) and (\ref{eq-18}) that $h^{e}$ and $g^{e}$ are central. So, by a small computation using Lemma \ref{lem0}(3) and (\ref{eq25}), we deduce from the identity above that
\begin{equation*}f^{e}(xyz)-f^{e}(xzy)=2(\alpha f-g+\beta h)(x)f^{o}(yz).\end{equation*}
As $f^{e}$ is central we have $f^{e}(xyz)-f^{e}(xzy)=f^{e}(zxy)-f^{e}(zyx)$. So, applying the identity above and Lemma \ref{lem0}(2) we deduce that
\begin{equation}\label{eq28}\phi(x)\psi_{y}(z)=-\phi(z)\psi_{y}(x)\end{equation}
for all $x,y,z\in S,$
where $\phi:=\alpha f-g+\beta h$ and $\psi_{y}(\cdot):=f^{o}(y\cdot)$ for each fixed $y\in S.$\\
On the other hand we deduce from (\ref{eq18}) and (\ref{eq24}) that $$f^{o}(xy)+f^{o}(x\sigma(y))=2 F_{2}(x)(\lambda f^{e}+h^{e})(y)$$ for all $x,y\in S.$
As $F_{2}\neq0$ and $\{f^{e},h^{e}\}$ is linearly independent we derive from the identity above that $f^{o}\neq 0$ on $S^{2}$. Then there exist $a\in S$ such that $\psi_{a}\neq0.$ So we derive from the identity (\ref{eq28}), applied for $y=a$, that $\phi=0$. Hence,
\begin{equation}\label{eq29}g=\alpha f+\beta h\end{equation} and then (\ref{eq1}) reduces to
$$f(x\sigma(y))=(2\alpha-\beta^{2})f(x)f(y)+l(x)l(y),\,\,x,y\in S,$$ where $l:=\beta f+h.$ Now, considering the odd parts of functions in (\ref{eq29}) and applying (\ref{eq23}) we obtain $(2\alpha-\beta^{2})f^{o}(xy)=$ for all $x,y\in S.$ As $f^{o}\neq 0$ on $S^{2}$ we derive that $2\alpha-\beta^{2}=0.$ So, the functional equation above becomes
\begin{equation*}f(x\sigma(y))=l(x)l(y),\,\,x,y\in S.\end{equation*}
Moreover, from (\ref{eq29}) and that $2\alpha-\beta^{2}=0$ we deduce that $g=\frac{\beta^{2}}{2}f+\beta h.$ Notice that $l\neq0$ because $l^{o}=\beta f^{o}+h^{o}=F_{2}\neq0.$ The result occurs in part (3), which completes the proof of Lemma \ref{lem4}.
\end{proof}

\section{Main results}
In view of Lemmas \ref{lem2}-\ref{lem4} we will discuss two principal cases according to whether $f^{e}$ and $h^{e}$ are linearly independent or not. We also take into account the fact that $f^{e}$ is zero on  $S^{2}$ or not.
\begin{thm}\label{thm1} The solutions of the functional equation (\ref{eq1}) such that $h^{e}$ and $f^{e}$ are linearly dependent and $f^{e}(S^{2})=\{0\}$ are the following triples of functions $f,g,h:S\to\mathbb{C}$:\\
(A) $f=0$, $g$ arbitrary and $h=0$.\\
(B) $f$ is any non-zero function such that $f(S^{2})=\{0\}$, $g=-\frac{\lambda^{2}}{2}f$ and $h=\lambda f$ where $\lambda\in\mathbb{C}$ is a constant.
\end{thm}
\begin{proof} If $f=0$, then (\ref{eq1}) implies
$h(x)h(y)=0$ for all $x,y\in S$, so $h=0$ and $g$ is arbitrary. The solution is of category (A). Assume that
$f\neq0$.\\
If $f^{e}=0$, then we get, by Lemma \ref{lem1}(2), a solution of category (B).\\ In what remains of the proof we assume that $f^{e}\neq0$. So, $h^{e}$ and $f^{e}$ being linearly dependent, we derive, according to Lemma \ref{lem3}, that
\begin{equation}\label{eq+31}g^{e}=-\frac{\lambda^{2}}{2} f^{e}\,\,\text{and}\,\,h^{e}=\lambda f^{e}\end{equation}
for some constant $\lambda\in\mathbb{C}$ and identities (\ref{eq15}) and (\ref{eq16}) are satisfied.
We get from (\ref{eq16}) that $h^{o}=0$ iff $f^{o}=0$ or $g^{o}=0.$ So, we split the discussion into the cases $h^{o}\neq0$ and $h^{o}=0.$\\
\underline{Case A}: $h^{o}\neq0$. Then, we get from (\ref{eq16}) that $f^{o}\neq0$, $g^{o}\neq0$ and there exists $(\alpha,\beta)\in\mathbb{C}\times\mathbb{C}\setminus\{(0,0)\}$ such that
$h^{o}=\alpha f^{o}+\beta g^{o}$. Proceeding as in part corresponding to $h\neq0$ in the proof of Lemma \ref{lem1} we deduce that there exists a constant $\delta\in\mathbb{C}\setminus\{0\}$ such that $g^{o}=-\frac{\delta^{2}}{2}f^{o}$ and $h^{o}=\delta f^{o}$.\\
Now, (\ref{eq2}) reduces to
\begin{equation}\label{eq++31}f^{o}(xy)+f^{o}(x\sigma(y))=-(\delta-\lambda)^{2}f^{o}(x)f^{e}(y)\end{equation} for all $x,y\in S.$\\ Let $x,y,z\in S$ be arbitrary. By applying (\ref{eq++31}) to the pair $(xy,z)$ and taking Lemma \ref{lem0}(3) into account, we get that $(\delta-\lambda)^{2}f^{o}(xy)f^{e}(z)=0$. So, $x,y,z\in S$ being arbitrary and $f^{e}\neq0$, we derive that $(\delta-\lambda)^{2}f^{o}(xy)=0$ for all $x,y\in S.$ Then, by multiplying (\ref{eq++31}) by $(\delta-\lambda)^{2}$ and seeing that $f^{o}\neq0$ and $f^{e}\neq0$, we derive that $\delta=\lambda.$ So that
\begin{equation}\label{eq+++31}g^{o}=-\frac{\lambda^{2}}{2} f^{o}\,\,\text{and}\,\,h^{o}=\lambda f^{o}.\end{equation}
Hence, we deduce from (\ref{eq+31}) and (\ref{eq+++31}) that $g=-\frac{\lambda^{2}}{2} f\,\,\text{and}\,\,h^{o}=\lambda f.$
Moreover, using Lemma \ref{lem0}(2) and that $f^{o}$ is odd, we infer from (\ref{eq++31}) that  $f^{o}(xy)=-f^{o}(x\sigma(y))=f^{o}(\sigma(y)x)=-f^{o}(y\sigma(x))=f^{o}(yx)=-f^{o}(xy)$, then $f^{o}(xy)=0$ for all $x,y\in S.$ So, $f(xy)=0$ for all $x,y\in S$. Hence, we obtain a solution of category (B).\\
\underline{Case B}: $h^{o}=0$. In this case we have $f^{o}=0$ or $g^{o}=0.$
\par If $f^{o}=0$ then $k=g^{o}$, where $k:S\to\mathbb{C}$ is the function defined in Lemma \ref{lem3}(2), i.e., $k:=-\frac{\lambda^{2}}{2} f^{o}+g^{o}+\lambda h^{o}.$ So (\ref{eq15}) becomes $g^{o}(x)f^{e}(y)+f^{e}(x)g^{o}(y)=0$ for all $x,y\in S.$ Hence, according to \cite[Exercise 1.1(b)]{H.St}, we get that $g^{o}=0$ because $f^{e}\neq0.$ Then, in view (\ref{eq+31}), we derive that $g=-\frac{\lambda^{2}}{2} f\,\,\text{and}\,\,h=\lambda f.$ Moreover, $f(xy)=f^{e}(xy)+f^{o}(xy)=0$ for all $x,y\in S.$ So, we obtain a solution of category (B).
\par If $g^{o}=0$ then $k=-\frac{\lambda^{2}}{2} f^{o}$ and (\ref{eq15}) becomes
\begin{equation}\label{eq++++31}2f^{o}(x\sigma(y))=-\lambda^{2}f^{o}(x)f^{e}(y)-\lambda^{2}f^{e}(x)f^{o}(y)\end{equation}
for all $x,y\in S.$ So, taking Lemma \ref{lem0}(3) and that $f^{e}(S^{2})=\{0\}$ into account, we derive from the identity above that $\lambda^{2}f^{o}(xy)=0$ for all $x,y\in S.$ Now, by multiplying (\ref{eq++++31}) by $\lambda^{2}$ and seeing that $f^{e}\neq0$ we derive, according to \cite[Exercise 1.1(b)]{H.St}, that $\lambda f^{o}=0.$\\
If $\lambda=0$ then, in view of (\ref{eq+31}), we get that $g=g^{e}+g^{o}=0$ and $h=h^{e}+h^{o}=0$. Moreover, we deduce from (\ref{eq++++31}) that $f^{o}(xy)=0$ for all $x,y\in S.$ Hence, $f(xy)=f^{e}(xy)+f^{o}(xy)=0$ for all $x,y\in S.$ So, we obtain a solution of category (B).\\
If $f^{o}=0$ then we go back to the first part in this case and we obtain again a solution of category (B).
\par Conversely, if $f,g$ and $h$ are of the forms (A)-(B) in Theorem \ref{thm1} we check that $f,g$ and $h$ satisfy the functional equation (\ref{eq1}), $f^{e}$ and $h^{e}$ are linearly dependent and $f^{e}(S^{2})=\{0\}.$ This completes the proof of Theorem \ref{thm1}.
\end{proof}
\begin{thm}\label{thm2} The solutions $f,g,h:S\to\mathbb{C}$ of the functional equation (\ref{eq1}) such that $h^{e}$ and $f^{e}$ are linearly dependent and $f^{e}(S^{2})\neq\{0\}$, are the following, where $m,\chi,\chi_{1},\chi_{2}:S\to \mathbb{C}$ are multiplicative functions such that  $\chi_{1}\neq\chi_{2}$, $k,\varphi,\psi:S\to \mathbb{C}$ are functions such that $k^{*}=-k$ and $k$ is $0$-addive, $\varphi$ is non-zero $\chi$-additive and $\psi$ is of type $(\chi,\sigma,\varphi)$-cosine-sine, and $\lambda,\mu,\eta\in\mathbb{C},\rho,c\in\mathbb{C}\setminus\{0\}$ are constants such that $\lambda^{2}-2\mu\neq0$ and $\rho^{2}=-\frac{1}{4(\lambda^{2}-2\mu)}.$\\
(A) $$f=\frac{1}{2(\lambda^{2}-2\mu)}(m+m^{*})+k,\,\,g=-\frac{\mu}{2(\lambda^{2}-2\mu)}(m+m^{*})+\lambda\rho(m-m^{*})-\frac{\eta^{2}}{2}k$$$$\,\,\text{and}\,\,h=\frac{\lambda}{\lambda^{2}-2\mu}(m+m^{*})-\rho(m-m^{*})+\eta k,$$
\par (i) If $m=m^{*}\neq0$ then $$\text{and}\,\,(\eta^{2}-2\lambda\eta+2\mu=0\,\,\text{and}\,\,\eta\neq0)\,\,\text{or}\,\,(\eta=\mu=0\,\,\text{and}\,\,\lambda\neq0).$$
\par (ii) If $m\neq m^{*}$ then  $k=0$.\\
(B) $$f=\frac{\chi_{1}-\chi_{2}}{2c},\,\,g=\frac{(2c-\lambda^{2})\chi_{1}+(2c+\lambda^{2})\chi_{2}}{4c}\,\,\text{and}\,\,h=\frac{\lambda(\chi_{1}-\chi_{2})}{2c}$$ such that $\chi_{1}^{*}=\chi_{1}$ and $\chi_{2}^{*}=\chi_{2}.$\\
(C) $$f=\varphi,\,\,g=\chi-\frac{\lambda^{2}}{2}\varphi\,\,\text{and}\,\,h=\lambda \varphi$$ such that $\chi\neq0$, $\chi^{*}=\chi$ and $\varphi^{*}=\varphi.$\\
(D) \begin{center}
$\begin{pmatrix} f \\
g \\ h \end{pmatrix}=\begin{pmatrix} 1 & 0 & 0 \\
-\frac{\lambda^{2}}{2}& 1 & i\lambda\\  \lambda & 0 &-i \end{pmatrix}\begin{pmatrix} \psi\\
\chi \\ \varphi \end{pmatrix}$\end{center}
such that $\psi^{*}=\psi$, $\chi^{*}=\chi$, $\varphi^{*}=-\varphi$ and $\varphi\neq0$.\\
(E) \begin{center}
$\begin{pmatrix} f \\
g \\ h \end{pmatrix}=\begin{pmatrix}
-2c^{2}& 4c^{2}& 0\\ \frac{1+4\lambda^{2}c^{2}}{4}& \frac{1-4\lambda^{2}c^{2}}{2} &-\lambda c \\ -2\lambda c^{2}& 4\lambda c^{2} &c\end{pmatrix}\begin{pmatrix} \chi_{1}+\chi_{1}^{*}\\
\chi \\ \chi_{1}-\chi_{1}^{*} \end{pmatrix}$\end{center}
such that $\chi^{*}=\chi$, $\chi\neq\chi_{1}$ and $\chi_{1}^{*}\neq\chi_{1}.$\\
(F) \begin{center}
$\begin{pmatrix} f \\
g \\ h \end{pmatrix}=\begin{pmatrix}
1&0& 1\\ -\frac{\lambda^{2}}{2} &-\lambda i&-\frac{\lambda^{2}}{2} \\ \lambda & i &\lambda\end{pmatrix}\begin{pmatrix} \psi\\
\varphi \\ k\end{pmatrix}$,\end{center}
where $\psi$ is of type $(0,\varphi)$-cosine-sine such that $\psi^{*}=\psi$ and $\varphi^{*}=-\varphi.$
\end{thm}
\begin{proof} Since $f^{e}(S^{2})\neq\{0\}$ we have $f^{e}\neq0$. As $h^{e}$ and $f^{e}$ are linearly dependent there exists a constant $\lambda\in\mathbb{C}$ such that $h^{e}=\lambda f^{e}$ and, according to Lemma \ref{lem2} there exists a constant $\mu\in\mathbb{C}$ such that (\ref{eq000}) is satisfied. We split the discussion into the following cases:\\
\underline{Case A}: $f^{o}=0$, $g^{o}=-\lambda h^{o}$ and the triple $(f^{e},g^{e}+\frac{\lambda^{2}}{2}f^{e},h^{o})$ satisfies (\ref{eq-11}). As in \cite[Section 4]{AjElq} we solve (\ref{eq-11}) according to whether $f^{e}$ and $h^{o}$ are linearly independent or not. As $f^{e}\neq0$ is even and $h^{o}$ is odd we split the discussion into the subcases $h^{o}=0$ and $h^{o}\neq0.$\\
\underline{Subcase A1}: $h^{o}=0.$ Then (\ref{eq-11}) reduces to the sine addition law below $$f^{e}(xy)=f^{e}(x)\left(g^{e}+\frac{\lambda^{2}}{2}f^{e}\right)(y)+f^{e}(y)\left(g^{e}+\frac{\lambda^{2}}{2}f^{e}\right)(x),\,\,x,y\in S.$$
Since $f^{e}\neq0$  we deduce, by applying \cite [Theorem 4.1(b)]{H.St}, that there exist a constant $c\in\mathbb{C}$ and two multiplicative functions $\chi_{1},\chi_{2}:S\to\mathbb{C}$ such that $g^{e}+\frac{\lambda^{2}}{2}f^{e}=\frac{\chi_{1}+\chi_{2}}{2}$ and $2cf^{e}=\chi_{1}-\chi_{2}.$ Since $f^{o}=0$ and $g^{o}=-\lambda h^{o}=0$ we get that $g+\frac{\lambda^{2}}{2}f=\frac{\chi_{1}+\chi_{2}}{2}$ and $2cf=\chi_{1}-\chi_{2},$ $f^{*}=f$ and $g^{*}=g.$ So, $\chi_{1}^{*}=\chi_{1}$ and $\chi_{2}^{*}=\chi_{2}.$
\par If $\chi_{1}\neq\chi_{2}$ then $c\neq0$ and we get a solution of category (B) by a small computation.
\par If $\chi_{1}=\chi_{2}$ then (\ref{eq-11}) reduce to $f^{e}(xy)=f^{e}(x)\chi(y)+\chi(x)f^{e}(x),$  where $\chi:=\chi_{1}.$ As $f^{e}(S^{2})\neq\{0\}$  we have $\chi\neq0.$ By putting $\varphi:=f^{e}$ we get a solution of part (C).\\
\underline{Subcase A2}: $h^{o}\neq0$. Then $f^{e}$ and $h^{o}$ are linearly independent because $f^{e}\neq0$ is even and $h^{o}\neq0$ is odd. According to Proposition \ref{prop1} we have one of the following subcases:\\
\underline{Subcase A2.1}: There exist a constant $\beta\in\mathbb{C}$ and two multiplicative functions $m,\chi:S\to\mathbb{C}$ such that
\begin{equation}\label{eq32}g^{e}+\frac{\lambda^{2}}{2}f^{e}=\chi,\,ih^{o}-\beta f^{e}=\varphi\,\,\text{and}\,\,\chi+i\beta h^{o}=m,\end{equation}
where $\varphi:S\to\mathbb{C}$ is $\chi$-additive.
\par If $\beta\neq0$ then, using second and third identities in (\ref{eq32}), we get that $f^{e}=c^{2}m-c^{2}\chi-c\varphi$ and $ih^{o}=c\,m-c\chi$ where $c:=\beta^{-1}\neq0$. Moreover, using first and second identities in (\ref{eq32}), and seeing that $f^{e}\circ\sigma=f^{e}$ and $g^{e}\circ\sigma=g^{e}$ we deduce that $\chi^{*}=\chi$ and $c\,m^{*}-\varphi^{*}=c\,m-\varphi$. Then $\varphi^{*}$ is $\chi$-additive, so is $\varphi^{*}-\varphi$ and $\varphi^{*}-\varphi=c\,m^{*}-c\,m\in \text{span}\{m,m^{*}\}.$ Then, according to Proposition \ref{prop001}, we derive that $\varphi^{*}-\varphi =0$, which implies that $c(m^{*}-m)=0$. This contradicts the fact that $c\neq0$ and $m\neq m^{*}$.\\
Hence $\beta=0$. So we derive from (\ref{eq32}) that
\begin{center}
\begin{equation}\label{eq330}\begin{pmatrix} f^{e} \\
g^{e}+\frac{\lambda^{2}}{2}f^{e} \\ ih^{o} \end{pmatrix}=\begin{pmatrix} 1 & 0 & 0 \\
0& 1 & 0\\ 0 & 0 & 1 \end{pmatrix}\begin{pmatrix} \psi\\
\chi \\ \varphi \end{pmatrix},\end{equation}\end{center}
which implies that
\begin{center}
\begin{equation}\label{eq3300}\begin{pmatrix} f^{e} \\
g^{e}\\ h^{o} \end{pmatrix}=\begin{pmatrix} 1 & 0 & 0 \\
-\frac{\lambda^{2}}{2}& 1 & 0\\ 0 & 0 & -i \end{pmatrix}\begin{pmatrix} \psi\\
\chi \\ \varphi \end{pmatrix},\end{equation}\end{center} where $\psi:S\to\mathbb{C}$ is of type $(\chi,\varphi)$-cosine-sine such that $\psi^{*}=\psi$, $\chi^{*}=\chi$, $\varphi^{*}=-\varphi$ and $\varphi\neq0$. As $h^{e}=\lambda f^{e}=\lambda\psi$ and $g^{o}=-\lambda h^{o}=i\lambda\varphi$ we get that
\begin{center}
\begin{equation}\label{eq33000}\begin{pmatrix} f^{o} \\
g^{o} \\ h^{e} \end{pmatrix}=\begin{pmatrix} 0 & 0 & 0 \\
0& 0 & i\lambda\\ \lambda & 0 & 0 \end{pmatrix}\begin{pmatrix} \psi\\
\chi \\ \varphi \end{pmatrix}.\end{equation}\end{center}
By adding (\ref{eq3300}) and (\ref{eq33000}) we get part (D).\\
\underline{Subcase A2.2}: There exist constants $\alpha\in\mathbb{C}\setminus\{0\}, \beta\in\mathbb{C}$ and a multiplicative function $\chi:S\to\mathbb{C}$ such that
\begin{equation}\label{eq034}-\alpha\,f^{e}+g^{e}+\frac{\lambda^{2}}{2}f^{e}=\chi\end{equation}
and the pair $(h^{o},\chi+2\,\alpha f^{e}+\dfrac{1}{2}\,\beta h^{o})$ satisfies the sine addition law, i.e.,
\begin{equation}\label{eq34}h^{o}(xy)=h^{o}(x)(\chi+2\,\alpha f^{e}+\dfrac{1}{2}\,\beta h^{o})(y)+(\chi+2\,\alpha f^{e}+\dfrac{1}{2}\,\beta h^{o})(x)h^{o}(y)\end{equation}
for all $x,y\in S.$ Then, according to \cite [Theorem 4.1(b)]{H.St}, there exist two multiplicative functions $\chi_{1},\chi_{2}:S\to\mathbb{C}$ and a constant $\rho\in\mathbb{C}$ such that
\begin{equation}\label{eq35}2\rho\,h^{o}=\chi_{1}-\chi_{2}\end{equation}
and
\begin{equation}\label{eq36}\chi+2\,\alpha f^{e}+\dfrac{1}{2}\,\beta h^{o}=\frac{\chi_{1}+\chi_{2}}{2}.\end{equation}
\par If $\chi_{1}=\chi_{2}$ then $\rho\,h^{o}=0$. So $\rho=0$ because $h^{o}\neq0$. By putting $m:=\chi_{1}$ we deduce from (\ref{eq34}) and (\ref{eq36}) that
\begin{equation}\label{eq-37}2\,\alpha f^{e}=-\chi+m-\dfrac{1}{2}\,\beta h^{o}\end{equation}
and that
\begin{equation}\label{eq37}h^{o}(xy)=h^{o}(x)m(y)+m(x)h^{o}(y),\,\,x,y\in S\end{equation}
Applying this to the pair $(\sigma(x),\sigma(y))$ and subtracting the identity obtained from (\ref{eq37}), we get that $h^{o}(x)(m(y)-m^{*}(y))=-h^{o}(y)(m(x)-m^{*}(x))$ for all $x,y\in S.$ As $h^{o}\neq0$ we deduce, according to \cite[Exercise 1.1(b)]{H.St}, that $m^{*}=m.$ Since $f^{e}$ is even and $h^{o}$ is odd we derive from (\ref{eq-37}) that
\begin{equation}\label{eq38}\chi^{*}-\chi=\beta h^{o}.\end{equation}
Since $h^{o}$ is $m$-additive (see (\ref{eq37})) so is $\beta h^{o}$. Then we derive from (\ref{eq38}), according to Proposition \ref{prop001}, that $\beta=0$ because $h^{o}\neq0.$ Hence, we deduce from (\ref{eq-37}) and (\ref{eq034}) that $f^{e}=\frac{m-\chi}{2\,\alpha}$ and $g^{e}+\frac{\lambda^{2}}{2}f^{e}=\frac{m+\chi}{2}.$ Substituting this in the functional equation (\ref{eq-11}) we derive, by a simple computation, that $f^{e}(xy)=f^{e}(xy)-h^{o}(x)h^{o}(y)$ for all $x,y\in S.$ So $h^{o}=0$, which contradicts the assumption on $h^{o}$ in this case.\\
Hence $\chi_{1}\neq\chi_{2}$. Then $\rho\neq0$ and it follows from (\ref{eq35}) that $h^{o}=\frac{\chi_{1}-\chi_{2}}{2\rho}$ which implies that $\chi_{1}+\chi_{1}^{*}=\chi_{2}+\chi_{2}^{*}$ because $h^{o}$ is an odd function. As $\chi_{1}\neq\chi_{2}$ we derive from the last identity, by applying
\cite[corllary 3.19]{H.St}, that $\chi_{2}=\chi_{1}^{*}.$\\
Now, considering the functional equation (\ref{eq-11}) and proceeding as in Subcase B.2 of the proof of \cite [Theorem 4.2]{AjElq2} and then as in Case A.4 of the proof of \cite [Theorem 4.5]{AjElq}, we derive, by putting $c:=\frac{1}{2\rho}$, that
\begin{center}
\begin{equation}\label{eq0330}\begin{pmatrix} f^{e} \\
g^{e}+\frac{\lambda^{2}}{2}f^{e} \\ ih^{o} \end{pmatrix}=\begin{pmatrix} -2c^{2} & 4c^{2} & 0 \\
\frac{1}{4}& \frac{1}{2} & 0\\ 0 & 0 & ic \end{pmatrix}\begin{pmatrix} \chi_{1}+\chi_{1}^{*}\\
\chi \\ \chi_{1}-\chi_{1}^{*}\end{pmatrix},\end{equation}\end{center}
which implies that
\begin{center}
\begin{equation}\label{eq40}\begin{pmatrix} f^{e} \\
g^{e}\\ h^{o} \end{pmatrix}=\begin{pmatrix} -2c^{2} & 4c^{2} & 0 \\
\frac{1+4\lambda^{2}c^{2}}{4}& \frac{1-4\lambda^{2}c^{2}}{2} & 0\\ 0 & 0 & c \end{pmatrix}\begin{pmatrix} \chi_{1}+\chi_{1}^{*}\\
\chi \\ \chi_{1}-\chi_{1}^{*} \end{pmatrix}.\end{equation}\end{center}
Since $c\neq0$, $f^{e}\circ\sigma=f^{e}$ and $h^{o}\circ\sigma=-h^{o}$ we get, by using (\ref{eq40}), that $\chi^{*}=\chi$, $\chi\neq\chi_{1}$ and $\chi_{1}^{*}\neq\chi_{1}.$
As $f^{o}=0$, $g^{o}=-\lambda h^{o}$ and $h^{e}=\lambda f^{e}$, we get that
\begin{center}
\begin{equation}\label{eq41}\begin{pmatrix} f^{o} \\
g^{o} \\ h^{e} \end{pmatrix}=\begin{pmatrix} 0 & 0 & 0 \\
0& 0 & -\lambda c\\ -2\lambda c^{2}& 4\lambda c^{2}& 0 \end{pmatrix}\begin{pmatrix} \chi_{1}+\chi_{1}^{*}\\
\chi \\ \chi_{1}-\chi_{1}^{*} \end{pmatrix}.\end{equation}\end{center}
 By adding (\ref{eq40}) and (\ref{eq41}) we get part (E).\\
 \underline{Subcase A2.3}: $f^{e}=F,\,\,\,\,g^{e}+\frac{\lambda^{2}}{2}f^{e}=-\dfrac{1}{2}\delta^{2}\,F+G+\delta\,H,\,\,\,\,ih^{o}=-\delta
\,F+H,$
where $\delta\in\mathbb{C}$ is a constant and $F,G,H:S\to\mathbb{C}$ are functions such that the triple $(F,G,H)$ is of the form in (\ref{eq330}) or (\ref{eq40}) with the same constraints on $(\psi,\chi,\varphi)$. Notice that if $(F,G,H)$ is of the form in (\ref{eq330}) then $G=\chi$ and $H=\varphi$ with $\chi^{*}=\chi$ and $\varphi^{*}=-\varphi\neq0$, if $(F,G,H)$ is of the form in (\ref{eq40}) then $G=\frac{1}{4}(\chi_{1}+\chi_{1}^{*})+\frac{1}{2}\chi$ and $H=ic(\chi_{1}-\chi_{1}^{*})$ with $\chi^{*}=\chi$ and $\chi_{1}\neq\chi_{1}^{*}$. Then $G^{*}=G$ and $H^{*}=-H\neq0$. Moreover $G^{*}-G=-\delta(H^{*}-H)$, then $\delta=0$. So we obtain a solution of part (D) or (E).\\
\underline{Case B}: $g^{e}=-\mu f^{e}$ for some constant $\mu\in\mathbb{C}$, $f^{o}$, $f^{e}$ and $h^{o}$ satisfy (\ref{eq013}), (\ref{eq14}) and (\ref{eq140}). By applying (\ref{eq140}) to the pair $(x,\sigma(y))$ we obtain
\begin{equation}\label{eq42}\begin{split}&f^{e}(xy)=(\lambda^{2}-2\mu)\left(f^{e}(x)f^{e}(y)+f^{o}(x)f^{o}(y)\right)\\
&\quad\quad\quad\quad\quad-\left(\lambda f^{o}(x)-h^{o}(x)\right)\left(\lambda f^{o}(y)-h^{o}(y)\right).\end{split}\end{equation}
We have two cases according to $\lambda^{2}-2\mu=0$ or $\lambda^{2}-2\mu\neq0$.\\
\underline{Subcase B1}: $\lambda^{2}=2\mu$. Then (\ref{eq42}) becomes
\begin{equation*}f^{e}(xy)=-\left(\lambda f^{o}(x)-h^{o}(x)\right)\left(\lambda f^{o}(y)-h^{o}(y)\right)\end{equation*} for all $x,y\in S.$ So, for $\psi:=f^{e}$  and $ -i\varphi:=\lambda f^{o}-h^{o}$ we get that $\psi$ is of type $(0,\varphi)$-cosine-sine, $\psi^{*}=\psi$ and $\varphi^{*}=-\varphi$. Now, using that $g^{e}=-\frac{^{2}}{2}f^{e}$, $g^{o}+\lambda h^{o}=\mu f^{o}$, and $\lambda^{2}=2\mu$ we obtain a solution of part (F) by putting $k:=f^{o}.$\\
\underline{Subcase B2}: $\lambda^{2}\neq2\mu$. By subtracting (\ref{eq140}) from (\ref{eq42}) we obtain
\begin{equation}\label{eq--43}f^{e}(xy)-f^{e}(x\sigma(y))=2(\lambda^{2}-2\mu)f^{o}(x)f^{o}(y)-2F(x)F(y)\end{equation}
for all $x,y\in S$, where
\begin{equation}\label{eq-43}F:=\lambda f^{o}-h^{o}.\end{equation}
Let $x,y,z\in S$ be arbitrary. By applying the identity above to the pair $(xy,z)$ and taking (\ref{eq013}) into account we get that
\begin{equation}\label{eq43}f^{e}(xyz)-f^{e}(xy\sigma(z))=-2F(xy)F(z).\end{equation}
On the other hand, according to \cite[Lemma 3.2]{ElqRed}, we get from (\ref{eq14}) that
\begin{equation}\label{eq44}f^{e}=\frac{1}{2(\lambda^{2}-2\mu)}(m+m^{*})\end{equation}
where $m:S\to\mathbb{C}$ is a multiplicative function.\\
We split the discussion into the cases $m\neq m^{*}$ and $m=m^{*}.$\\
\underline{Subcase B2.1}: $m\neq m^{*}$. By substituting (\ref{eq44}) in (\ref{eq43}) we obtain
\begin{equation}\label{eq45} \left(m(xy)-m^{*}(xy)\right)\left(m(z)-m^{*}(z)\right)=-4(\lambda^{2}-2\mu)F(xy)F(z)\end{equation}
for all $x,y,z\in S.$ As $m\neq m^{*}$ we derive from (\ref{eq45}) that $F\neq0$ on $S^{2}$. So there exist $x_{0},y_{0}\in S$ such that $F(x_{0}y_{0})\neq0$. Putting $(x,y)=(x_{0},y_{0})$ in (\ref{eq45}) we derive that there exists a constant $\rho\in\mathbb{C}\setminus\{0\}$ such that
\begin{equation}\label{eq46}F=\rho(m-m^{*}).\end{equation}
\par On the other hand $f^{o}=0$. Indeed, if $f^{o}\neq0$ we derive, by substituting (\ref{eq44}) and (\ref{eq46}) in (\ref{eq--43}) that there exist $a,b\in\mathbb{C}$ such that $f^{o}=a\,m+b\,m^{*}$. As $f^{o}$ is odd and $m\neq m^{*}$ we deduce, according to \cite[Theorem 3.18(b)]{H.St}, that $b=-a\neq0$. Hence, by using (\ref{eq013}), we get that $am(y)\,m-am^{*}(y)\,m^{*}=0$ for each fixed $y\in S.$ So, according to \cite[Theorem 3.18(b)]{H.St}, we deduce that $am(y)=0$ for all $y\in S$, which contradicts that $a\neq0$ and $m\neq m^{*}.$\\
Hence, by using (\ref{eq46}), (\ref{eq-43}) and  (\ref{eq000}) we obtain
\begin{equation}\label{eq47}g^{o}=\lambda\rho(m-m^{*})\,\,\text{and}\,\,h^{o}=-\rho(m-m^{*}).\end{equation}
Similarly, using (\ref{eq44}) and that $g^{e}=-\mu f^{e}$ and $h^{e}=\lambda f^{e}$ we obtain
\begin{equation}\label{eq48}g^{e}=-\frac{\mu}{2(\lambda^{2}-2\mu)}(m+m^{*})\,\,\text{and}\,\,h^{e}=\frac{\lambda}{2(\lambda^{2}-2\mu)}(m+m^{*}).\end{equation}
Moreover, by substituting (\ref{eq44}) and (\ref{eq46}) in (\ref{eq--43}) and taking into account that $f^{o}=0$, a small computation using \cite[Theorem 3.18(b)]{H.St} shows that $\rho^{2}=-\frac{1}{4(\lambda^{2}-2\mu)}.$ Now,  using (\ref{eq44}), (\ref{eq47}), (\ref{eq48}) and that any complex-valued function on $S$ is the sum of its even and odd parts we obtain a solution of category (A)(ii).\\
\underline{Subcase B2.2}: $m=m^{*}$. Then
\begin{equation}\label{eq49}f^{e}=\frac{1}{\lambda^{2}-2\mu}m,\,g^{e}=-\frac{\mu}{\lambda^{2}-2\mu}m\,\,\text{and}\,\,h^{e}=\frac{\lambda}{\lambda^{2}-2\mu}m\end{equation}
with $m\neq0$, and (\ref{eq--43}) reduces to
\begin{equation}\label{eq50}(\lambda^{2}-2\mu)f^{o}(x)f^{o}(y)=\left(\lambda f^{o}(x)-h^{o}(x)\right)\left(\lambda f^{o}(y)-h^{o}(y)\right).\end{equation}
We have the following subcases:\\
\underline{Subcase B2.2.1}: $f^{o}\neq0$. Then we get from (\ref{eq50}) that $f^{o}=\gamma(\lambda f^{o}-h^{o})$ for some constant $\gamma\in\mathbb{C}\setminus\{0\}.$ Hence
\begin{equation}\label{eq51} (\lambda\gamma-1)f^{o}=\gamma h^{o}.\end{equation}
\par If $\lambda\gamma-1=0$. Then we deduce from (\ref{eq51}) that $h^{o}=0$ because $\gamma\neq0$. So (\ref{eq50}) implies that $\mu=0$ because $f^{o}\neq0$ and $\lambda\neq0$. Hence, taking (\ref{eq000}) into account, we get that $g^{o}=0$. So, by putting $k:=f^{o}$ and using (\ref{eq49}), we obtain $f=\frac{1}{\lambda^{2}}m+k,\,g=0,\,\,\text{and}\,\,h=\frac{1}{\lambda}m$. We get a solution of category (A)(i).
\par If $\lambda\gamma-1\neq0$. Then (\ref{eq51}) implies that
\begin{equation}\label{eq52}h^{o}=\eta f^{o},\end{equation} where $\eta:=\frac{\gamma}{\lambda\gamma-1}\in\mathbb{C}\setminus\{0\}.$ Moreover, by substituting (\ref{eq52}) in (\ref{eq50}) and using that $f^{o}\neq0$, a small computation shows the identity $\eta^{2}-2\lambda\eta+2\mu=0.$ Then, using (\ref{eq000}), we deduce that $g^{o}=(\mu-\lambda\eta)f^{o}$. So that
\begin{equation}\label{eq53}g^{o}=-\frac{\eta^{2}}{2} f^{o}.\end{equation}
Now, by putting $k:=f^{o}$ and using (\ref{eq49}), (\ref{eq52}) and (\ref{eq53}), we obtain a solution of category (A)(i).\\
\underline{Subcase B2.2.2}: $f^{o}=0$. Then the identities (\ref{eq50}) and (\ref{eq000}) imply that $g^{o}=h^{o}=0.$  So, we go back to subcase A1 and get a solution of category (B) or (C).
\par Conversely, if $f,g$ and $h$ are of the forms (A)-(F) in Theorem \ref{thm2} we check that $f,g$ and $h$ satisfy the functional equation (\ref{eq1}), $f^{e}$ and $h^{e}$ are linearly dependent and $f^{e}(S^{2})\neq\{0\}.$ This completes the proof of Theorem \ref{thm2}.
\end{proof}
\begin{thm}\label{thm3} The solutions of the functional equation (\ref{eq1}) such that $h^{e}$ and $f^{e}$ are linearly independent are the following triples of functions $f,g,h:S\to\mathbb{C}$ where $\chi,\mu,\chi_{1},\chi_{2},\chi_{3}:S\to\mathbb{C}$ are non-zero even multiplicative functions such that $\chi_{1},\chi_{2},\chi_{3}$ are different, $\chi\neq\mu$ and $\chi\neq0$, $\Phi_{0},\Psi_{0}:S\to\mathbb{C}$ are non-zero functions and $\psi,\varphi:S\to\mathbb{C}$ are non-zero even functions
such that $\varphi$ is $\chi$-additive, $\Phi_{0}$ is $0$-additive, $\psi$ is of type $(\chi,\varphi)$-cosine-sine and $\Psi_{0}$ is of type $(0,\sigma,\Phi_{0})$-cosine-sine such that $\Phi_{0}^{*}=\Phi_{0}$ and $\Psi_{0}^{*}=\Psi_{0}$ in parts (F)-(H),
and $k:S\to\mathbb{C}$ a $0$-addive function such that $k^{*}=-k$.\\
$\beta,\delta\in\mathbb{C},\alpha,\lambda,\rho,c\in\mathbb{C}\setminus\{0\}$ are constants such that $2\alpha\lambda^{2} \rho(2-\rho)=1.$\\
(A)
$$\begin{pmatrix} f \\
g \\ h \end{pmatrix}=\begin{pmatrix} 1 & 0 & 0 \\
0 & 1 & 0 \\ 0 & 0 & 1 \end{pmatrix}\begin{pmatrix} \psi\\
\chi \\ \varphi \end{pmatrix}.$$
(B) $$\begin{pmatrix} f \\
g \\ h \end{pmatrix}=\begin{pmatrix} c^{2} & -c^{2} & -c \\
0 & 1 & 0 \\ c & -c & 0 \\ \end{pmatrix}\begin{pmatrix} \mu\\
\chi \\ \varphi \end{pmatrix}.$$
(C) $$\begin{pmatrix} f \\
g \\ h \end{pmatrix}=\begin{pmatrix} -a & a & -ab \\
\frac{1}{2} & \frac{1}{2} & -\frac{b}{2} \\ 0 & 0 & 1\\ \end{pmatrix}\begin{pmatrix} \mu\\
\chi \\ \varphi  \end{pmatrix}.$$
(D)
$$\begin{pmatrix} f \\
g \\ h \end{pmatrix}=\begin{pmatrix} \alpha\rho & \alpha(2-\rho) & -2\alpha \\
\frac{1}{4}\,\rho & \frac{1}{4}\,(2-\rho) & \frac{1}{2} \\ \frac{1}{2\,\lambda} & -\frac{1}{2\,\lambda} & 0 \\ \end{pmatrix}\begin{pmatrix} \chi_{1}\\
\chi_{2} \\ \chi_{3} \end{pmatrix}.$$
(E)
$$\begin{pmatrix} f \\
g \\ h \end{pmatrix}=\begin{pmatrix} 1& 0 & 0 \\
-\frac{\delta^{2}}{2} & 1 & \delta \\ -\delta & 0 & 1 \\ \end{pmatrix}\begin{pmatrix} f_{0}\\
g_{0} \\ h_{0} \end{pmatrix},$$
where the functions $f_{0},g_{0},h_{0}:S\to\mathbb{C}$ are of the forms (A)-(D) with the same constraints.\\
(F)
$$\begin{pmatrix} f \\
g \\ h \end{pmatrix}=\begin{pmatrix} 1& 0\\
-\frac{\beta^{2}}{2}& \beta\\ -\beta&  1 \\ \end{pmatrix}\begin{pmatrix} \Psi_{0}\\
\Phi_{0}\end{pmatrix}+\begin{pmatrix} k \\
-\frac{\beta^{2}}{2}k\\ -\beta k \end{pmatrix}.$$
(G) $$\begin{pmatrix} f \\
g \\ h \end{pmatrix}=\begin{pmatrix}
c^{2}& -c\\ \frac{c\beta(2-c\beta)}{2}&\frac{c\beta^{2}}{2}\\
c(1-c\beta)&c\beta\end{pmatrix}\begin{pmatrix}\mu\\
\Phi_{0}\end{pmatrix}+\begin{pmatrix} k \\
-\frac{\beta^{2}}{2}k\\ -\beta k \end{pmatrix}.$$
\end{thm}
\begin{proof} According to Lemma \ref{lem4} we have the following cases:\\
\underline{Case A}: $f$, $g$ and $h$ are even functions such that $\{f,g,h\}$ is linearly independent and $(f,g,h)$ satisfies (\ref{eq17}), i.e.,
\begin{equation*}f(xy)=f(x)g(y)+g(x)f(y)+h(x)h(y),\,x,y\in S.\end{equation*}
By applying \cite [Theorem 4.2]{AjElq2}, we obtain easily solutions of category (A)-(E), where $f_{0},g_{0},h_{0}:S\to\mathbb{C}$ are of the forms (A)-(D) with the same constraints.\\
\underline{Case B}: $\{f^{e},g^{e},h^{e}\}$ is linearly dependent and $(f^{e},g^{e},h^{e})$ satisfies (\ref{eq--18}), i.e.,
\begin{equation*}f^{e}(xy)=f^{e}(x)g^{e}(y)+g^{e}(x)f^{e}(y)+h^{e}(x)h^{e}(y),\,\,x,y\in S,\end{equation*}
$f^{o}(xy)=0$ for all $x,y\in S$, and
\begin{equation}\label{eq54} g^{e}=\frac{\beta^{2}}{2}f^{e}+\beta h^{e},\,\,g^{o}=-\frac{\beta^{2}}{2}f^{o}\,\,\text{and}\,\,h^{o}=-\beta f^{o}\end{equation} for some constant $\beta\in\mathbb{C}.$ Then, according to \cite [Theorem 4.2]{AjElq2} we have one of the following possibilities:\\
(1) $f^{e}=\psi,\,\,g^{e}=\chi\,\,\text{and}\,\,h^{e}=\varphi$. By using the first identity in (\ref{eq54}) and \cite [Proposition 3.6]{AjElq2} we derive that $\chi=0$. Then
 \begin{equation}\label{eq--55}f^{e}=\psi,\,\,g^{e}=0\,\,\text{and}\,\,h^{e}=\varphi\end{equation} and  $g^{o}=h^{o}=0$. So, by putting $\Psi_{0}:=f^{e}$, $\Phi_{0}:=h^{e}$ and $k:=f^{o}$ we get a solution of category (F) corresponding to $\beta=0$.\\
(2) $f^{e}=c^{2}\,\mu-c^{2}\,\chi-c\,\varphi,\,\,g^{e}=\chi\,\,\text{and}\,\,h^{e}=c\,\mu-c\,\chi$, where $\chi,\mu:S\to\mathbb{C}$ different multiplicative functions and $\varphi:S\to\mathbb{C}$ a non-zero $\chi$-additive function, and $c\in\mathbb{C}\setminus\{0\}$ a constant. Then $\chi^{*}=\chi$, $\mu^{*}=\mu$ and $\varphi^{*}=\varphi.$ In view of first identity in (\ref{eq54}) we get that
\begin{equation}\label{eq-55}c\beta^{2}\varphi=\left(c^{2}\beta^{2}+2c\beta\right)\mu+\left(c^{2}\beta^{2}-2c\beta-2\right)\chi.\end{equation}
Then $c\beta^{2}\varphi\in \text{span}\{\mu,\chi\}$.
Since $\varphi$ is $\chi$-additive, so is $c\beta^{2}\varphi$, then we derive, by using Proposition \ref{prop001}, that $c\beta^{2}\varphi = 0$. So that $\beta=0$ because $c\neq 0$ and $\varphi \neq 0.$ It follows, by using (\ref{eq-55}), that $\chi=0.$ Hence, taking (\ref{eq54}) and the expressions of $f^{e}$, $g^{e}$ and $h^{e}$ into account, we deduce that
\begin{equation}\label{eq056}f^{e}=c^{2}\,\mu-c\,\varphi,\,\,g^{e}=0\,\,\text{and}\,\,h^{e}=c\,\mu\end{equation} and $g^{o}=h^{o}=0$. Moreover $\mu\neq0$ because $\mu\neq\chi.$ So, by putting $\Phi_{0}:=h^{e}$ and $k:=f^{o}$ we get a solution of category (G) corresponding to $\beta=0$.\\
(3) $f^{e}=-a\,\mu+a\,\chi-ab\,\varphi,\,\,g^{e}=\frac{1}{2}\mu+\frac{1}{2}\chi-\frac{b}{2}\,\varphi\,\,\text{and}\,\,h^{e}=\varphi$, where $\chi,\mu:S\to\mathbb{C}$ different multiplicative functions and $\varphi:S\to\mathbb{C}$ a non-zero $\chi$-additive function and $a,b\in\mathbb{C}\setminus\{0\}$ are constants such that $1+ab^{2}=0.$ Then $\chi^{*}=\chi$, $\mu^{*}=\mu$ and $\varphi^{*}=\varphi.$ By using the first identity in (\ref{eq54}) and the identity $1+ab^{2}=0$ we obtain after a small computation
$$\frac{(b+\beta)^{2}}{2b}\varphi=\frac{1+a\beta^{2}}{2}\mu+\frac{1-a\beta^{2}}{2}\chi.$$
Proceeding as in (2) we prove that $\beta=-b$, $\chi=0$ and $\mu\neq0$. So that \begin{equation}\label{eq0056}f^{e}=-a\,\mu-ab\,\varphi,\,\,g^{e}=\frac{1}{2}\mu-\frac{b}{2}\,\varphi\,\,\text{and}\,\,h^{e}=\varphi,\end{equation} and using (\ref{eq54}) we get that  $g^{o}=-\frac{b^{2}}{2}f^{o}\,\,\text{and}\,\,h^{o}=bf^{o}.$ So, by putting $\Phi_{0}:=h^{e}$ and $k:=f^{o}$ we get a solution of category (H) corresponding to $\beta=-b$.\\
(4) $f^{e}=\alpha\rho\,\chi_{1}+\alpha(2-\rho)\chi_{2}-2\,\alpha\,\chi_{3},\,\,g^{e}=\frac{1}{4}\rho\,\chi_{1}+\frac{1}{4}(2-\rho)\chi_{2}+\frac{1}{2}\:\chi_{3}$\\
and $h^{e}=\frac{1}{2\lambda}\chi_{1}-\frac{1}{2\lambda}\chi_{2},$ where $\alpha,\lambda,\rho\in\mathbb{C}\setminus\{0\}$ are constants such that $2\alpha\lambda^{2} \rho(2-\rho)=1$ and ;
$\chi_{1},\,\chi_{2},\,\chi_{3}:S\to\mathbb{C}$ are three different
multiplicative functions. Then $\chi_{1}=\chi_{1}^{*},\,\,\chi_{2}=\chi_{2}^{*}\,\,\text{and}\,\,\chi_{3}=\chi_{3}^{*}.$ By using first identity in (\ref{eq54}) we obtain
\begin{equation}\label{eq56}\begin{split}&\left(\frac{\rho}{4}-\frac{\rho}{2}\alpha\beta^{2}-\frac{\beta}{2\lambda}\right)\chi_{1}
+\left(\frac{1}{2}-\alpha\beta^{2}-\frac{\rho}{4}+\frac{\rho}{2}\alpha\beta^{2}+\frac{\beta}{2\lambda}\right)\chi_{2}\\
&\quad\quad\quad+\left(\frac{1}{2}+\alpha\beta^{2}\right)\chi_{3}=0.\end{split}\end{equation}
Since coefficients of the linear combination in (\ref{eq56}) can not be zero at the same time and the multiplicative functions $\chi_{1},\,\chi_{2},\,\chi_{3}$ are different, we deduce, according to \cite[Theorem 3.18(b)]{H.St}, that  $\chi_{1}=0\,\,\text{or}\,\,\chi_{2}=0\,\,\text{or}\,\,\chi_{3}=0.$
\par If $\chi_{1}=0$ then $\chi_{2}\neq0$ and $\chi_{3}\neq0$. As $\chi_{1}\neq\chi_{2}$ we derive from (\ref{eq56}), using \cite[Theorem 3.18(b)]{H.St}, that
$\frac{1}{2}-\alpha\beta^{2}-\frac{\rho}{4}+\frac{\rho}{2}\alpha\beta^{2}+\frac{\beta}{2\lambda}=0$ and $\frac{1}{2}+\alpha\beta^{2}=0,$ which implies that $2\alpha\beta^{2}=-1$ and $\beta=-\lambda(2-\rho).$ So that $-\rho=2-\rho$ because $2\alpha\lambda^{2} \rho(2-\rho)=1$ , which is a contradiction. Similarly, supposing $\chi_{2}=0$ or $\chi_{3}=0$ we get a contradiction.
\par Hence, the present possibility does not arise.\\
(5) \begin{equation}\label{eq57}f^{e}=F_{0},\,\,g^{e}=-\frac{\delta^{2}}{2}F_{0}+G_{0}+\delta H_{0}\,\,\text{and}\,\,h^{e}=-\delta F_{0}+H_{0},\end{equation} where $\delta\in\mathbb{C}$ is a constant and the functions $F_{0},G_{0},H_{0}:S\to\mathbb{C}$ are of the forms (A)-(D) with the same constraints to the ones in \cite [Theorem 4.2(1)-(4)]{AjElq2}. So $F_{0},G_{0}$ and $H_{0}$ are even and the triple $(F_{0},G_{0},H_{0})$ satisfies the cosine-sine functional equation. Moreover, a small computation, using first identity in (\ref{eq54}), shows that $G_{0}=\frac{\gamma^{2}}{2}F_{0}+\gamma H_{0}$, where $\gamma:=\beta-\delta$. So, we deduce as above that $F_{0},G_{0},H_{0}$ are of the forms in (\ref{eq--55}) or (\ref{eq056}) or (\ref{eq0056}), and by analogy with (1), (2) and (3) we have the following possibilities:\\
(5-1): $F_{0},G_{0},H_{0}$ are of the forms in (\ref{eq--55}), i.e., $$F_{0}=\psi,\,\,G_{0}=0,\,\,H_{0}=\varphi\:\:\text{and}\,\,\gamma=0.$$
Then, applying (\ref{eq57}), we get that
$$f^{e}=\psi,\,\,g^{e}=-\frac{\beta^{2}}{2}\psi+\beta\varphi\,\,\text{and}\,\,h^{e}=-\beta\psi+\varphi.$$ Hence, taking (\ref{eq54}) into account, obtain a solution of category (F) by putting $k:=f^{o}.$\\
(5-2): $F_{0},G_{0},H_{0}$ are of the forms in (\ref{eq056}), i.e.,
$$F_{0}=c^{2}\,\mu-c\,\varphi,\,\,G_{0}=0,\,\,H_{0}=c\,\mu\,\,\text{and}\,\,\gamma=0,$$
which implies, by using (\ref{eq57}), that
\begin{equation*}f^{e}=c^{2}\,\mu-c\,\varphi,\,\,g^{e}=\frac{c\,\beta(2-c\,\beta)}{2}\mu+\frac{c\,\beta^{2}}{2}\varphi
\,\,\text{and}\,\,h^{e}=c(1-c\,\beta)\mu+c\,\beta\varphi.\end{equation*}
Combining this with (\ref{eq54}) we obtain a solution of category (G) by putting $k:=f^{o}.$\\
(5-3): $F_{0},G_{0},H_{0}$ are of the forms in (\ref{eq0056}), i.e.,
$$F_{0}=-a\,\mu-ab\,\varphi,\,\,G_{0}=\frac{1}{2}\mu-\frac{b}{2}\,\varphi,\,\,H_{0}=\varphi\,\,\text{and}\,\,\gamma=-b.$$
Then, by a small computation (\ref{eq57}) leads to
\begin{equation*}f^{e}=-a\,\mu-ab\,\varphi,\,\,g^{e}=\frac{1+a(b+\beta)^{2}}{2}\mu+\frac{ab\beta^{2}}{2}\varphi
\,\,\text{and}\,\,h^{e}=a(b+\beta)\mu+ab\,\beta\varphi.\end{equation*}
So, by putting $c:= ab$, $k:=f^{o}$ and taking (\ref{eq54}) into account, we obtain a solution of category (G).\\
\underline{Case C}: There exists a constant $\beta\in\mathbb{C}$ such that $\beta f^{o}+h^{o}\neq0$ and
\begin{equation}\label{eq58}g=\frac{\beta^{2}}{2}f+\beta h,\end{equation} and the pair $(f,l)$ satisfies the functional equation (\ref{eq+17}), i.e.,
\begin{equation}\label{eq59}f(x\sigma(y))=l(x)l(y),\,x,y\in S,\end{equation}
where
\begin{equation}\label{eq60}l:=\beta f+h.\end{equation}
Let $x,y,z\in S$ be arbitrary. By applying (\ref{eq+17}) to the pair $(\sigma(x),y)$ and $(x,\sigma(y))$ we get the identities $f^{*}(xy)=l^{*}(x)l(y)$ and $f(xy)=l(x)l^{*}(y)$, which gives by subtraction $2f^{o}(xy)=l(x)l^{*}(y)-l^{*}(x)l(y).$
When we apply this to the pair $(x,\sigma(y))$ we obtain $2f^{o}(x\sigma(y))=l(x)l(y)-l^{*}(x)l^{*}(y).$
By addition of the two last identities we deduce that $f^{o}(xy)+f^{o}(x\sigma(y))=2l^{o}(x)l^{e}(y).$
So that
\begin{equation*}f^{o}(xyz)+f^{o}(xy\sigma(z))=2l^{o}(xy)l^{e}(z)\,\,\text{and}\,\,f^{o}(xyz)+f^{o}(x\sigma(yz))=2l^{o}(x)l^{e}(yz).\end{equation*}
As $l^{o}=\beta f^{o}+h^{o}\neq0$ and $l^{e}=\beta f^{e}+h^{e}\neq0$ because $\{f^{e},h^{e}\}$ is linearly independent, we deduce from the identity above, by using Lemma \ref{lem0}(3), that $l^{o}(xy)=0$ and $l^{e}(yz)=0$. Hence, $x,y,z\in S$ being arbitrary, $l(xy)=0$ for all $x,y\in S.$ So, by putting $\Psi_{0}:=f$ and $\Phi_{0}:=l$ we get that $$\Phi_{0}(xy)=\Phi_{0}(x)\cdot0+0\cdot\Phi_{0}(y),$$  for all $x,y\in S,$  and in view of (\ref{eq59}), we have $$\Psi_{0}(x\sigma(y))=\Psi_{0}(x)\cdot0+0\cdot\Psi_{0}(y)+\Phi_{0}(x)\Phi_{0}(y),$$ for all $x,y\in S.$ Notice that $f$ and $l$ are non-zero because $f^{e}\neq0$ and $l^{e}\neq0$, so are $\Psi_{0}$ and $\Phi_{0}$. Moreover, by using (\ref{eq60}) and (\ref{eq58}) we derive that $$h=-\beta\Psi_{0}+\Phi_{0}\,\,\text{and}\,\,g=-\frac{\beta^{2}}{2}\Psi_{0}+\beta_{0}.$$ So, we obtain a solution of category (F) corresponding to $k=0$.
\par Conversely, if $f,g$ and $h$ are of the forms (A)-(I) in Theorem \ref{thm3} we check that $f,g$ and $h$ satisfy the functional equation (\ref{eq1}), and that $f^{e}$ and $h^{e}$ are linearly independent. This completes the proof of Theorem \ref{thm3}.
\end{proof}



\begin{thebibliography}{1}
\bibitem{Acz} J. Acz\'{e}l,  \textit {Lectures on functional equations and their applications}, Mathematics in Sciences and Engineering, vol. 19. Academic Press, New York, xx+510 pp, 1966.
\bibitem{AjElq}  O. Ajebbar and E. Elqorachi, \textit {The Cosine-Sine functional equation on a semigroup with. an involutive automorphism}, Aequat. Math. \textbf{91} (2017), 1115-1146. https://doi.org/10.1007/s00010-017-0512-9
\bibitem{AjElq2} O. Ajebbar and E. Elqorachi, \textit {Cosine-sine functional equation on semigroups}.
https://doi.org/10.48550/arXiv.2306.04666
\bibitem{CKN} J.K. Chung, Pl. Kannappan and C.T. Ng, \textit {A generalization of
	the Cosine-Sine functional equation on groups}, Linear Algebra and
Appl. \textbf{66}  (1985), 259-277.
\bibitem{E1} B. Ebanks, \textit {The sine addition and subtraction formulas on semigroups}, Acta Math. Hungar. \textbf{164} (2021), 533-555. https://doi.org/10.1007/s10474-021-01149-3
\bibitem{Eba} B. Ebanks, \textit {The cosine and sine addition and subtraction formulas on semigroups}, Acta Math. Hungar. \textbf{165} (2021), 37-354. https://doi.org/10.1007/s10474-021-01167-1
\bibitem{E2} B. Ebanks, \textit { Around the Sine Addition Law and d’Alembert’s Equation on Semigroups}, Results Math \textbf{77}(11) (2022). https://doi.org/10.1007/s00025-021-01548-6
\bibitem{EB} B. Ebanks, \textit { Some Levi-Civita Functional equations on Semigroups}, Results Math \textbf{77} (154) (2022). https://doi.org/10.1007/s00025-022-01705-5
\bibitem{EB1} B. Ebanks, \textit {The Cosine-Sine Functional Equation on Semigroups}, Annales Mathematicae Silesianae \textbf{36}(1) (2022), 30-52. https://doi.org/10.2478/amsil-2021-0012
\bibitem{EB2} B. Ebanks, \textit {The Solution of the Cosine–Sine Functional Equation on Semigroups}, Results Math \textbf{78}(151) (2023). https://doi.org/10.1007/s00025-023-01930-6
\bibitem{ElqRed}  E. Elqorachi and A. Redouani, \textit {Solutions and stability of a variant of Wilson's functional equation}, Proyecciones (Antofagasta) \textbf{37}(2) (2018), 317-344. http://dx.doi.org/10.4067/S0716-09172018000200317
\bibitem{stet1} B. Ebanks and H. Stetkær, \textit {d’Alembert’s other functional equation on monoids with an
involution},  Aequat. Math. \textbf{89} (2015), 187-206.  https://doi.org/10.1007/s00010-014-0303-5
\bibitem{H.St}  H. Stetk\ae r, \textit {Functional Equations on Groups}, World Scientific Publishing Company, Singapore 2013.
\bibitem{H.St1} H. Stetk\ae r, \textit {A Levi–Civita functional equation on semigroups}, Aequat. Math. \textbf{96} (2022), 115-127. https://doi.org/10.1007/s00010-021-00865-z	
\bibitem{vin} E. Vincze, \textit {Eine allgemeinere Methode in der Theorie der Funktionalgleichungen}, II,
Publ. Math. Debrecen \textbf{9} (1962), 314-323.
\end{thebibliography}


\end{document}